\documentclass{article}
\usepackage{amsmath}
\usepackage{amsfonts}
\usepackage[alphabetic]{amsrefs}
\usepackage{a4wide}

\usepackage{pgfplots}
\pgfplotsset{width=10cm,compat=1.9}

\usetikzlibrary{
        pgfplots.external,
    }
\usepackage{float}

\usepackage{enumitem}

\usepackage{mathtools}

\usepackage{cases}
\usepackage{amsthm}

\usepackage{hyperref}

\theoremstyle{theorem}
\newtheorem{thm}{Theorem}[section]
\newtheorem{lm}[thm]{Lemma}

\newtheorem{proposition}[thm]{Proposition}

\theoremstyle{remark}

\theoremstyle{definition}
\newtheorem{dft}[thm]{Definition}

\numberwithin{equation}{section}

\newcommand{\dt}{\partial_t}
\newcommand{\dx}{\partial_x}
\newcommand{\idx}{\,\mathrm{d}x}

\newcommand{\mrm}[1]{\mathrm{#1}}

\newcommand{\norm}[2]{\Vert #1 \Vert_{#2}}

\newcommand{\eps}{\varepsilon}
\newcommand{\R}{\mathbb{R}}
\newcommand{\N}{\mathbb{N}}

\definecolor{Turk}{rgb}{0,0.7,0.4}

\newcommand{\EEE}{\color{black}}

\begin{document}
    
\title{Plastic limit of a viscoplastic Burgers equation \\-- A toy model for sea-ice dynamics}
\author{Xin Liu\footnote{Department of Mathematics, Texas A\&M University, College Station, TX 77843-3368, USA. Email: \textsf{xliu23@tamu.edu}}\space,\space Marita Thomas\footnote{Department of Mathematics and Computer Science, Freie Universität Berlin, Arnimallee 9, 14195 Berlin, Germany. Email: \textsf{marita.thomas@fu-berlin.de}}\space, and Edriss S. Titi\footnote{Department of Applied Mathematics and Theoretical Physics, University of Cambridge, Cambridge CB3 0WA UK; Department of Mathematics, Texas A\&M University, College Station, TX 77843-3368, USA; also Department of Computer Science and Applied Mathematics, Weizmann Institute of Science, Rehovot 76100, Israel. Emails: \textsf{Edriss.Titi@maths.cam.ac.uk} \; \textsf{titi@math.tamu.edu}}}

\maketitle

\begin{abstract}
We study the plastic Burgers equation in one space dimension, i.e., the Burgers equation featuring an additional term formally given by the $p$-Laplacian with $p=1$, or rather, by the multivalued subdifferential of the total variation functional. 
Our study highlights that the  interplay of the advection term with the stresses given by the multivalued $1$-Laplacian is a crucial feature of this model. Eventhough it is an interesting model in itsef, it can also be regarded as a one-dimensional version of the momentum balance of Hibler's model for sea-ice dynamics. Therein, the stress tensor is given by a term with similar properties as the $1$-Laplacian in order to account for plastic effects of the ice.  
For our analysis we start out from a viscoplastic Burgers equation, i.e., a suitably regularized version of the plastic Burgers equation with a small regularization parameter $\eps>0$. 
For the viscoplastic Burgers equation, we construct a global $\mathrm{BV}$-solution. In the singular limit  $\eps\to0$ we deduce the existence of a $\mathrm{BV}$-solution for the plastic Burgers equation. In addition we show that the term arising as the limit of the regularized stresses is indeed  related to an element of the subdifferential of the total variation functional. 

    \bigskip
    {\noindent\bf Keywords:} Burgers equation; (visco)plastic fluid; singular limit; total variation flow with advection; subdifferential of non-smooth, convex  potential; Hibler's sea-ice model.\\
    {\noindent\bf MSC2020:} 
    35D30, 
    35D40, 
    35Q35, 
    35Q86, 
    49J45, 
    49J53, 
    76A05, 
    86A05. 
\end{abstract}

%
\section{Introduction}
%
As a feature of climate change the rapidly progressing melting of sea-ice is observed. 
This has severe impact on the global (especially northern hemispherical) atmospheric and oceanic currents as well as the local and global ecosystem (Thomas and  Dieckmann \cite{seaice-second-edition}), but it does open the opportunity for marine routes connecting the Atlantic and Pacific oceans, which will drastically reduce the marine shipping time. A widely used model, describing the dynamics and thermodynamics of sea-ice, was introduced by Hibler \cite{Hibler1979}. This model was further discussed and developed by a large number of scientists in the field. To name a few, we refer to Hunke \cite{Hunke2001}, Dansereau et al.\ \cite{dansereauMaxwellElastobrittleRheology2016}, Schreyer et al.\ \cite{schreyerElasticdecohesiveConstitutiveModel2006}, and Wilchinsky and Feltham \cite{wilchinskyRheologyDiscreteFailure2012};  see also Pritchard \cite{Pritchard2001} for a summary of models. 

As mentioned in \cite{liuWellPosednessHiblerDynamical2022a}, Hibler's original sea-ice model poses the challenging and fundamental problem of well-posedness of solutions to the nonlinearly coupled parabolic system, featuring the momentum balance as a subdifferential inclusion due to the non-smoothness of the plastic potential, with a hyperbolic system of conservation laws for the height and the compactness of the sea-ice. In particular, the loss of hyperbolicity and linear ill-posedness was discussed in Gray and Killworth \cite{grayStabilityViscousPlasticSea1995}, and Gray \cite{grayLossHyperbolicityIllposedness1999}. A regularization of Hibler's model that retains its original  coupled parabolic-hyperbolic  character 
was introduced, and its local well-posedness was  investigated, in \cite{liuWellPosednessHiblerDynamical2022a}. Therein the regularization used for the analysis  was also motivated by similar choices made for numerical simulations in e.g.\ \cite{Mehlmann2017a,Mehlmann2017,MehKor20SIDT}.  Independently, the model  was also studied analytically in \cite{brandtRigorousAnalysisDynamics2022} with an additional  parabolic regularization of the hyperbolic conservation laws.  
Hereafter, in relatively short time a comparably large body of work has been developed by different authors addressing different aspects on regularized versions of Hibler's model, also relying on the parabolic regularization of the model and the powerful tool of maximal parabolic regularity, cf.\ e.g.\ \cite{BrHi23TPSH,BiBrHi24RAIP,Bra25WPHP,DiDi25GEUH}. As has been observed in e.g.\  \cite{Hunke1997} the advection term is at least one order of magnitude smaller than the acceleration term and therefore it has been  omitted in numerical simulations based on the elastic-viscous-plastic (EVP) sea-ice model, which emerged from Hibler's model in order to overcome numerical difficulties. Global well-posedness of the EVP model  has been recently established in \cite{BLTT25GWEP} with an inviscid Voigt regularization.  
Most recently, \cite{DeGmHi25SLHS} studies the singular limit from viscoplastic to plastic rheology of a gradient flow, without advection as it has been suggested in \cite{BLTT25GWEP}, involving the stress tensor proposed by Hibler in \cite{Hibler1979}. In two space dimensions the appearance of the symmetrized velocity gradient requires sophisticated convergence analysis in the space of bounded deformations. Yet, neglecting the advection term in the momentum balance of Hibler's model allows the authors in \cite{DeGmHi25SLHS} to study this problem as an $L^2$-gradient flow of a convex, positively $1$-homogeneous functional. Due to its non-smoothness the gradient flow evolution can be formulated in terms of a variational inequality in $L^2$. This notion of solution was already introduced and studied for the total variation functional in e.g.\ \cite{Giga1998,Giga2010} and e.g.\ \cite{ABCM01DPTV,ABCM01MTVF,BeCaNo02TVFR,AndMaz05TVF}.  
\par
In contrast, in this paper, our goal is to make progress in the investigation of the sea-ice momentum balance with  
the plastic rheology of Hibler's model in the unregularized, plastic version \emph{in presence of the advection term}. This significantly changes the structure of the equation from a gradient flow to a damped Hamiltonian system. 
This is mainly due to the fact that the dynamics of sea-ice is described in the Eulerian frame involving the convective time derivative, which in general does not satisfy a gradient flow structure.  
To better understand the mathematical nature of this problem, in particular the interplay of the advection term with the non-smooth stress term, we will also make a number of simplifications to Hibler's original momentum balance, which we explain in the following.
The momentum balance in Hibler's model in the time interval $[0,T]$ and the sea-ice domain $\Omega\subset\R^2$ takes the form
\begin{subequations}
\begin{equation}
\label{mombal-Hibler}
m(\boldsymbol{\dot u}+\boldsymbol{u}\cdot\nabla\boldsymbol{u})-\mathrm{div}\,\boldsymbol{\sigma}_\mathrm{H}
=\boldsymbol{f}(\boldsymbol{u})\,,    
\end{equation}
where $\boldsymbol{u}:[0,T]\times\Omega\to\R^2$ is the unknown ice velocity, $m$ is the ice mass, and $\boldsymbol{f}(\boldsymbol{u})$ comprises external forces, such as forces due to wind and ocean stresses as well as the Coriolis force. The stress tensor $\boldsymbol{\sigma}_\mathrm{H}$ in Hibler's model from \cite{Hibler1979,Hibl77VSLS} reads
\begin{equation}
\label{stress-Hibler}
\begin{split}
\boldsymbol{\sigma}_\mathrm{H}
&:=\frac{2}{\mathrm{e}^2}\zeta(\boldsymbol{e},P)\boldsymbol{e}+\big(
\zeta(\boldsymbol{e},P)-\eta(\boldsymbol{e},P)\big)
\mathrm{tr}\,\boldsymbol{e}\,\mathbb{I}_{2\times2}
-\frac{P}{2}\mathbb{I}_{2\times2}
\,,\;\;\text{with }\\
\zeta(\boldsymbol{e},P)&:=\frac{P}{2\triangle(\boldsymbol{e})}\,,\;\;\;
\eta(\boldsymbol{e},P)
:=\frac{1}{\mathrm{e}^2}\zeta(\boldsymbol{e},P)\,,\;\;\text{ and}\\
\triangle(\boldsymbol{e})&:=
\Big(
(e_{11}^2+e_{22}^2)\big(
1+\frac{1}{\mathrm{e}^2}
\big)
+\frac{4}{\mathrm{e}^2}e_{12}^2
+2e_{11}e_{22}\big(
1-\frac{1}{\mathrm{e}^2}
\big)
\Big)^{1/2}\,.
\end{split}
\end{equation}
\end{subequations}
Herein, $\boldsymbol{e}:=
\frac{1}{2}(\nabla \boldsymbol{u}+\nabla \boldsymbol{u}^\top)$ is the symmetrized gradient of the velocity field $\boldsymbol{u}$ with components  $e_{ij},$ $i,j=1,2$, $P$ is the prescribed ice pressure, and the constant $\mathrm{e}>0$ denotes the ratio of the principal axes of the ellipse that describes the plastic yield surface in terms of $\triangle(\boldsymbol{e})$. Moreover, $\mathbb{I}_{2\times2}\in\R^{2\times 2}$ is the identity matrix.
It can be readily checked that the stress tensor $\boldsymbol{\sigma}_\mathrm{H}$ is obtained as the derivative $\partial_{\boldsymbol{e}}\hat\psi_{\mathrm{H}}(\boldsymbol{e})$ of potential 
\begin{equation}
\label{psi-Hibler}
\hat\psi_{\mathrm{H}}(\boldsymbol{e})
:=\psi_{\mathrm{H}}(\boldsymbol{e})
-\frac{P}{2}\,\mathbb{I}_{2\times 2}:\boldsymbol{e}\quad
\text{ with }\;
\psi_{\mathrm{H}}(\boldsymbol{e}):=P\cdot\triangle(\boldsymbol{e})\,.
\end{equation}
Note that  $\psi_{\mathrm{H}}:\R^{2\times2}\to[0,\infty)$ is convex and positively $1$-homogeneous with $\psi_{\mathrm{H}}(\boldsymbol{0})=0$. In particular, $\psi_\mathrm{H}$ is not classically differentiable at $\boldsymbol{e}=0,$ so that the relation for $\boldsymbol{\sigma}_\mathrm{H}$ rather has to be understood as a subdifferential inclusion for the non-smooth potential $\psi_\mathrm{H},$ i.e.,
\begin{equation*}
   \big(\boldsymbol{\sigma}_\mathrm{H}+\frac{P}{2}\,\mathbb{I}_{2\times 2}\big)\in\partial\psi_\mathrm{H}(\boldsymbol{e})\,.
\end{equation*}
Furthermore note that choosing the ratio $\mathrm{e}$ of the principal axes as $\mathrm{e}=1,$ turns the ellipse into a ball as it results with this choice in 
\begin{equation*}
\triangle(\boldsymbol{e})=\sqrt{2}\,\big|\boldsymbol{e}\big|\,.
\end{equation*}
Also note that the characteristic properties of $\psi_\mathrm{H}$ are retained for the choice $\mathrm{e}=1$, i.e., also for the plastic yield surface being a ball, 
the potential $\psi_\mathrm{H}$ is convex, positively $1$-homogeneous, and non-smooth for $\boldsymbol{e}=0$. 
\par 
Exactly this very non-smoothness of $\psi_\mathrm{H}$ in combination with the advection term makes the mathematical analysis of the momentum balance \eqref{mombal-Hibler} challenging. The inherent features are already visible in one space dimension. Hereby, the external loadings and the effect of the pressure do not cause additional obstructions. For our mathematical study of the problem we shall therefore make the following simplifications: 
We restrict the analysis to one space dimension,  
 set the parameters $m=P=1$ in \eqref{mombal-Hibler}, $\mathrm{e}=1$ in \eqref{stress-Hibler} and also neglect the external loadings $\boldsymbol{f}(\boldsymbol{u})$ as well as the  the linear term $-\frac{P}{2}\mathbb{I}_{2\times2}:\boldsymbol{e}$ in $\hat\psi_\mathrm{H}$ from \eqref{psi-Hibler}.  With these simplifications we arrive at the following toy model for Hibler's sea-ice momentum balance: 
 \begin{subequations}
 \label{CauchyP-burgers-plastic}
 \begin{align}
\label{burgers-plastic}
    \dt u + \dx({u^2}/{2}) - \dx \sigma&=0\quad\text{ in } D\quad\;\text{with }\\
 \label{burgers-plastic-subd}   
    \sigma&\in \partial\psi(\partial_xu)\,,\quad\text{where}\\
      \label{def-psi}
\psi:\mathbb{R}\to[0,\infty),\quad \psi(e)&:=|e|\,.  
\end{align}
This is the Burgers equation featuring an additional stress term that arises as in \eqref{burgers-plastic-subd}  from the subdifferential of the convex, positively $1$-homogeneous potential $\psi$ from \eqref{def-psi}. As explained above, this corresponds to the special case of the yield surface being a ball in Hibler's model in more than one space dimension.   
\par 
Classical results in literature, cf.\ e.g.\ \cite{burgersMathematicalModelIllustrating1948,dafermosHyperbolicConservationLaws2000,Evans2010}, are established for the Burgers equation, i.e.\ with $\psi\equiv0$ in \eqref{def-psi}, on $[0,\infty)\times\R$.  In this line we also choose  
\begin{equation}
D:=[0,\infty)\times\R    
\end{equation}
in \eqref{burgers-plastic}. Accordingly, equation \eqref{burgers-plastic} is complemented by the initial condition
\begin{equation}
\label{burgers-IC}
u(0,x)=u_\mathrm{in}(x)\quad\text{for all }x\in\R\,.    
\end{equation}
\end{subequations}
Since $\psi$ from \eqref{def-psi}  is convex, relation \eqref{burgers-plastic-subd} has to be understood in the sense of subdifferentials of convex potentials, i.e., for all $e\in\mathbb{R}$ there holds  
\begin{equation}
\label{def-subdiff}
\sigma\in\partial\psi(e)
\quad\Leftrightarrow\quad
\forall\,\tilde e\in\mathbb{R}:\;
\sigma\cdot(\tilde e-e)\leq\psi(\tilde e)-\psi(e)\,,
\end{equation}
and $\partial\psi$ is set-valued. For the potential $\psi$ from \eqref{def-psi} we have in particular
\begin{equation}
\label{subdiff-psi}
\begin{split}
\partial\psi(e)=\mathrm{Sign}(e)
&:=\left\{
\begin{array}{cl}
\big\{\frac{e}{|e|}\big\}&\text{if }e\not=0\,,\\[1ex]
{[-1,1]}&\text{if }e=0\,,
\end{array}
\right.\\[1.5ex]
&\,=
\left\{
\begin{array}{cl}
\{-1\}&\text{if }e<0\,,\\[0.5ex]
{[-1,1]}&\text{if }e=0\,,\\[0.5ex]
\{1\}&\text{if }e>0\,.
\end{array}
\right.
\end{split}
\end{equation}
Although the plastic Burgers equation \eqref{burgers-plastic} has the potential to apply also for other physical scenarios, we here give an interpretation in the context 
of sea-ice. The advection term in \eqref{burgers-plastic} is responsible for the flow of sea-ice. 
The stress tensor obtained from the non-smooth potential describes the non-Newtonian, plastic rheology of sea-ice. Possible shock formation from the nonlinear advection can be seen as the ridging of sea-ice, where the ice flow compresses locally to form a jump. The (formation of) plateaus, where $\partial_xu=0,$ can be seen as (the formation of) a large block of ice mass, moving at the same speed. The interaction of shocks and plateaus then corresponds to the collision of ice ridges and large ice blocks. 
%
\subsection{Burgers equation, total variation flow, and very singular diffusion}
%
Our approach to the study of Cauchy problem \eqref{CauchyP-burgers-plastic}  is heavily inspired by classical studies of the Burgers  equation, which was proposed in  \cite{burgersMathematicalModelIllustrating1948} as a model for turbulence and is itself an important equation in the study of hyperbolic conservation laws. We refer e.g.\ to Bressan \cite{bressanGlobalSolutionsSystems1992}, Dafermos \cite{dafermosHyperbolicConservationLaws2000a}, and Lax \cite{laxHyperbolicPartialDifferential2006a} for monographs on hyperbolic PDEs. In particular, the notion of $\mathrm{BV}$-solutions, which serves as a weak solution concept for first order hyperbolic conservation laws, will also be applied to  \eqref{CauchyP-burgers-plastic}. Regularizing the Burgers equation by the $2$-Laplacian stemming from the quadratic viscous potential $\psi_2(e):=\frac{1}{2}|e|^2$ leads to the viscous Burgers equation. Due to its parabolicity it has smooth solutions and its vanishing-viscosity limit leads to Oleinik's entropy condition    \cite{hoffSharpFormOleinik1983,oleinikDiscontinuousSolutionsNonlinear1957}, which together with the Rankine-Hugoniot condition ensures the uniqueness of $\mathrm{BV}$-solutions for the inviscid Burgers equation. 
 Kurganov and Rosenau investigated in \cite{kurganovEffectsSaturatingDissipation1997} 
the equation 
\begin{equation}\label{eq:SaB}
	\dt u + \dx (u^2/2) - \dx \biggl( \dfrac{\partial_xu}{\sqrt{\vert \partial_xu\vert^2 + 1}} \biggr) = 0\,.
\end{equation}
This is a regularized version of \eqref{burgers-plastic}, where $\psi$ from \eqref{def-psi} is replaced by 
\begin{equation*}
\tilde\psi(e):=\sqrt{|\partial_xu|^2+1}\,.
\end{equation*}
It is observed in \cite{kurganovEffectsSaturatingDissipation1997}  that  large gradients will lead to jumps. That is, the parabolicity of \eqref{eq:SaB} is too weak to prevent shock formation for given near-shock data. Unique smooth solutions with small data are also constructed in the same paper. See also \cite{kurganovBurgersTypeEquationsNonmonotonic1998,goodmanBreakdownBurgerstypeEquations1999,bertschHyperbolicPhenomenaStrongly1992,Ryk00DSSS}, for related studies. 
\par
However, these studies do not apply to \eqref{CauchyP-burgers-plastic}, where the stress term degenerates to an element of the multi-valued subdifferential if $\partial_xu=0$. 
Such a degeneracy also appears in the study of the total variation flow; note that $\|\partial_xu\|_{L^1(\R)}$ arising from  \eqref{def-psi}  results by relaxation in the total variation of a function $u\in \mathrm{BV}(\R)$. The total variation flow is formally given by the evolution equation 
\begin{equation}\label{eq:gf0}
	\dt u = \mathrm{div} \, \biggl( \dfrac{\nabla u}{|\nabla u|} \biggr)\,,
\end{equation}
and also called the very singular diffusion equation in  \cite{Giga2010,Giga1998}. It can be formulated as an $L^2$-gradient flow for the total variation functional. The non-smoothness of the total variation functional at $\mathrm{D}_xu=0$  thus leads to a formulation in terms of a variational inequality, in view of the definition \eqref{def-subdiff}, naturally replacing the multivalued subdifferential of the total variation functional, cf.\ \cite{Giga2010,Giga1998} as well as e.g.\ \cite{ABCM01DPTV,ABCM01MTVF,BeCaNo02TVFR,AndMaz05TVF}.  
Yet, again, for the plastic Burgers equation \eqref{CauchyP-burgers-plastic}, this gradient flow structure is not available due to the presence of the advection term, which rather results in a damped Hamiltonian system.  
\par 
%
\subsection{Outline and results of this work}
\label{Sec:Outline-Results}
%
Now we outline our study of \eqref{CauchyP-burgers-plastic} in this work. Firstly, in Section \ref{Ex-BV},  
we prove the existence of global $\mathrm{BV}$-solution for the Cauchy problem \eqref{CauchyP-burgers-plastic} of the plastic Burgers equation. Motivated by the vanishing viscosity approach for the Burgers equation \cite{hoffSharpFormOleinik1983,oleinikDiscontinuousSolutionsNonlinear1957} we achieve this  by a viscous approximation of the plastic Burgers equation \eqref{burgers-plastic} in terms of a small parameter $\eps>0,$ in combination with an $\eps$-dependent regularization of the plastic potential. For $\eps>0$ this regularized model is called the viscoplastic Burgers equation and it reads as follows  
\begin{align*}
\partial_tu^\eps+\partial_x(u^\eps)^2/2
-\partial_x\big(\sigma^\eps+\eps\partial_xu^\eps\big)&=0\quad\text{in }D\,,\quad\text{ where}\\
\sigma^\eps&=\partial_x\psi_\eps(\partial_xu^\eps)\quad\text{ with }
\psi_\eps(e):=\sqrt{|e|^2+\eps^2}  \,,  
\end{align*}
complemented by suitable initial conditions and far-field boundary conditions. Thanks to the regularizations made, for each $\eps>0$ fixed,  we deduce in Theorem \ref{Thm-Ex-Approx-Sol} the existence of a unique, global $H^2$-solution, which also satisfies a weak formulation in the sense of integral solutions, cf.\ \cite[Sec.\ 3.4.1]{Evans2010}, i.e., $u^\eps:D=[0,\infty)\times\R\to\R$ satisfies 
\begin{equation}
\label{weak-ibvp-approx-intro}
\int_0^\infty\int_{\mathbb{R}}\Big(
u^\varepsilon\partial_t\phi+\big(
\frac{(u^\varepsilon)^2}{2}-\varepsilon\partial_xu^\varepsilon-\sigma^\varepsilon
\big)\partial_x\phi
\Big)\,\mathrm{d}x\,\mathrm{d}t
+\int_{\mathbb{R}}u^\varepsilon_{\mathrm{in}}\phi(0)\,\mathrm{d}x=0
\end{equation}
for all $\phi\in \mathrm{C}^1_{\mathrm{c}}([0,\infty)\times\R)$. 
 We deduce an energy-dissipation balance and first apriori estimates, uniformly in $\eps>0$  in Proposition \ref{EDB-UniBds}. They provide sufficient compactness, i.e., the existence of a subsequence $(u^\eps,\sigma^\eps)_\eps$ and of a limit pair $(u,R),$ such that in particular
 \begin{align*}
u^\eps&\to u\;\;\text{in }L^2(D)\,,\\
\sigma^\eps
&\overset{*}{\rightharpoonup}R\;\;\text{in }L^\infty(D)\,,\\
\eps\partial_xu^\eps&\to0\;\;\text{in }L^2(D)\,,
 \end{align*}
 to pass to the plastic limit as $\eps\to0$ in the weak formulation \eqref{weak-ibvp-approx-intro} and in the energy-dissipation balance of the viscoplastic problem in Section \ref{Sec-BV-sol-plast}, see \eqref{weak-BV-uR} and \eqref{cdt:ene-ineq} below. In addition, we show in Section \ref{Sec:Limit-Oleinik} that also Oleinik's entropy condition is satisfied both for the solutions of the viscoplastic Burgers equation and for the plastic limit, see \eqref{cdt:oleinik} below. In fact, also here, similar to the classical, inviscid Burgers equation, Oleinik's entropy condition implies that a solution can only decrease across a jump discontinuity. 
 
 Due to the obtained convergence result the plastic limit problem will therefore feature the pair $(u,R),$ where $R\in[-1,1]$ a.e.\ in $D$ is the weak-$*$ limit of the regularized stresses $\sigma^\eps$. It is a major task to further establish a relation of $R$ to an element of the subdifferential of a functional related to $\psi$ from \eqref{def-psi} evaluated in $u$, i.e., a relation akin to \eqref{burgers-plastic-subd}. In fact, we establish for the pair $(u,R)$ in Section \ref{Sec:CharR-Detail} with Theorems \ref{Char-Ri} and \ref{Char-Rii} that condition \eqref{compat-R-intro} below holds true with the time-integrated total variation functional
 \begin{equation}
 \label{def-Psi-intro} 
 \begin{split}
 &\Psi:X\to[0,\infty]\,,\quad 
 \Psi(v):=\left\{\begin{array}{cl}
\int_0^T|\mathrm{D}_xv|(\R)\,\mathrm{d}t
&\text{if }v\in Y\,,\\
\infty&\text{otherwise}\,,
 \end{array}\right.\\
 &\text{with }\;\;
 X:=L^2((0,T);L^2(\R))
 \quad\text{and}\quad 
 Y:=X\cap L^1((0,T);\mathrm{BV}(\R))\,,
 \end{split}
 \end{equation}
and where $|\mathrm{D}_xv|(\R)$ denotes the total variation with respect to the variable $x\in\R$ on the spatial domain $\R$ of the function 
$v\in Y$. We point out that, in view of  $|R|\leq1$ a.e.\ in $[0,\infty)\times\R$ the identification relation \eqref{compat-R-intro} corresponds to the characterization of the elements of the subdifferential of the total variation functional as already established in earlier works in the static setting or in the case of the total variation flow, cf.\ e.g.\ \cite{ChGoNo15FPSC,BeCaNo02TVFR,AndMaz05TVF,Giga1998,Giga2010}. In order to establish relation \eqref{compat-R-intro} we will use tools from convex analysis. 
To apply a Minty-type argument to conclude relation \eqref{compat-R-intro} in Theorems \ref{Char-Ri} and \ref{Char-Rii} will require to establish the validity of a weak formulation for the plastic limit problem \eqref{CauchyP-burgers-plastic} for a  larger class of test functions than used in \eqref{weak-BV-uR}, namely for the test functions in the space 
$V:=L^\infty(0,T;L^\infty(\R)\cap\mathrm{BV}(\R))$. Indeed we verify in Section \ref{sec:impro-est}, Proposition \ref{UniBDsImpro} the validity of additional uniform bounds for the approximating sequence $(u^\eps,\sigma^\eps)_\eps,$ which allow it to conclude the regularity properties \eqref{reg-limit-intro} for the limit pair $(u,R)$ and thus qualify the choice of $V$. Since smooth functions approximate $\mathrm{BV}$-functions only with respect to strict convergence the validity of the weak formulation in $V$ cannot be directly concluded from \eqref{weak-BV-uR} by density. Instead, to obtain the result in Section \ref{Sec:Weak-V}, Theorem \ref{Weak-V}  we start the argument once more on the level of the approximating problems, conclude the existence of limits for the terms $(\partial_tu^\eps)_\eps,$ 
$(\partial_x(u^\eps)^2/2)_\eps$ and 
$(\partial_x \sigma^\eps+\eps\partial_x^2u^\eps)_\eps$ in the dual space $V^*$ thanks to the additional uniform estimates,  and show that they can be identified  in the sense of distributions with the corresponding distributional derivatives of the limit $(u,R)$. 

Summarizing the above discussion, a $\mathrm{BV}$-solution for the Cauchy problem \eqref{CauchyP-burgers-plastic} of the plastic Burgers equation is defined as follows: 
\begin{dft}[$\mathrm{BV}$-solution of the plastic Burgers equation]
\label{def:bv-sol}
	A pair $ (u, R) $ is a $\mathrm{BV}$-solution to the Cauchy problem \eqref{CauchyP-burgers-plastic} of the plastic Burgers equation, if,
	\begin{enumerate}
		\item $ (u, R) $ satisfies \eqref{burgers-plastic} in the sense of distributions, i.e., for any $ \phi \in \mathrm{C}_c^\infty([0,\infty)\times\R) $ it  holds 
			\begin{equation}
			\label{weak-BV-uR}
				\int_0^\infty\int_{\R} \biggl( u\, \dt \phi + (u^2/2) \dx \phi - R\, \dx \phi \biggr)  \,\mathrm{d}x\,\mathrm{d}t
                +\int_{\R}u_\mathrm{in}\phi(0)\,\mathrm{d}x = 0\,
			\end{equation}
            with the initial datum $u(t=0)=u_\mathrm{in}$;
		\item the pair $(u,R)$ has the following regularity 
        \begin{subequations}
        \label{reg-limit-intro}
        \begin{align}
        u &\in L^\infty (0,\infty; L^\infty(\mathbb R)
        \cap \mathrm{BV}(\mathbb R)) \cap \mathrm{C}([0,\infty);L^1(\mathbb R))\cap L^\infty(0,\infty;L^2(\mathbb R))
        \,,\\ 
        R &\in L^\infty (0,\infty; L^\infty(\mathbb R) \cap \mathrm{BV}(\mathbb R))
        \;\text{ and }\; 
        \vert R \vert \leq 1 \text{ a.e.\ in }[0,\infty)\times\R\,;
        \end{align}
        \end{subequations}
        \item the energy-dissipation estimate is satisfied for all $t>0$ 
			\begin{equation}\label{cdt:ene-ineq}
				\dfrac{1}{2}\norm{u(t)}{L^2(\mathbb R)}^2 
				+ \int_0^t |\mathrm{D}_x u|(\R) \,\mathrm{d}s 
				\leq \dfrac{1}{2} \norm{u_\mrm{in}}{L^2(\mathbb R)}^2\,,
			\end{equation}
            where $|\mathrm{D}_x u|(\R)$ denotes the total variation of $u$ with respect to the variable $x\in\R$;  
        \item Oleinik's entropy condition is satisfied in the sense of distributions, i.e., 
			\begin{equation}\label{cdt:oleinik}
				\mathrm{D}_x u(x,t) < \frac{1}{t} \quad \text{for all } t > 0,
			\end{equation}
            where $ \mathrm{D}_x u $ is the distributional derivative of $ u $ with respect to the variable $x$;
				\item For all $T>0$ the compatibility condition 
        \begin{equation}
        \label{compat-R-intro}
\exists\,\zeta\in\partial\Psi(\mathrm{D}_xu)
\quad\text{and}\quad
-\mathrm{D}_xR=\zeta
\;\text{ in the sense of distributions}
        \end{equation}
        holds true for the functional 
        $\Psi:L^2(0,T;L^2(\R))\to[0,\infty]$ from \eqref{def-Psi-intro}.
	\end{enumerate}
\end{dft}
Now we are in the position to state the main result in this paper:
\begin{thm}[Existence of $\mathrm{BV}$-solutions for the plastic Burgers equation]\label{thm:bv-sol}
	Assume that the initial datum $ u_\mrm{in} \in L^\infty(\mathbb R) \cap \mathrm{BV}(\mathbb R) \cap L^2(\mathbb R) $ satisfies the compatibility condition
	\begin{equation}\label{ini:dt-u}
		\norm{\dt u(t=0)}{L^1(\mathbb R)} = \norm{ - \dx (u_\mrm{in}^2/2) + \dx (\dx u_\mrm{in} / \vert \dx u_\mrm{in}\vert)}{L^1(\mathbb R)} < \infty\,.
	\end{equation}
	Then there exists a $\mathrm{BV}$-solution of Cauchy problem \eqref{CauchyP-burgers-plastic} of the plastic Burgers equation in the sense of Definition \ref{def:bv-sol}.
\end{thm}
%
\section{Existence of $\mathrm{BV}$-solutions} 
\label{Ex-BV}
%
In the following, we construct a weak solution for Cauchy problem \eqref{CauchyP-burgers-plastic} in the sense 
of $\mathrm{BV}$-solutions, cf.\ Def.\ \ref{def:bv-sol},  as the asymptotic limit $\varepsilon\to0$ of the solutions to the following Cauchy problems with a viscous regularization scaled by a parameter $\varepsilon>0$: For all $ \varepsilon \in (0,1) $ consider 
\begin{subequations}
\label{ibvp-approx}
\begin{equation}
\label{eq:b-p-app}
    \dt u^\varepsilon + \dx({(u^\varepsilon)^2}/{2}) - \dx (\dfrac{\dx u^\varepsilon}{\sqrt{\vert \dx u^\varepsilon \vert^2 + \varepsilon^2}}) - \varepsilon \partial_{xx} u^\varepsilon = 0\qquad \text{in} ~ (0,\infty)\times\mathbb R,
\end{equation}
complemented by the following far-field boundary conditions
\begin{eqnarray}
\label{bc:far-field}
	\lim_{|x|\rightarrow \infty} u^\varepsilon(t,x) &=& 0\qquad\text{for all }t\in[0,\infty)\,,
\end{eqnarray}
and complemented by the following initial condition
\begin{equation}
u(0,x)=u_\mathrm{in}^\varepsilon(x)\quad\text{ for all }x\in\mathbb{R}\,.
\end{equation}
\end{subequations}
For any $\varepsilon>0$ fixed the following existence result holds true and for the readers' convenience we give the proof 
in Section \ref{Sec-Pf-Ex-Approx-Sol}: 
\begin{thm}[Existence of approximate solutions]
\label{Thm-Ex-Approx-Sol}
Keep $\varepsilon>0$ fixed and assume that the initial datum has the regularity $u_\mathrm{in}^\varepsilon\in H^2(\R)$ and that it also satisfies \eqref{bc:far-field} 
for $t=0$. Then, for all $T>0$ there is a weak solution $u^\varepsilon:[0,T]\times\R\to\R$ of the approximate 
Cauchy problem \eqref{ibvp-approx}, i.e.,  $u^\eps$ satisfies 
\begin{equation}
\label{weak-ibvp-approx}
\begin{split}
\int_0^\infty\int_{\mathbb{R}}\Big(
u^\varepsilon\partial_t\phi+\big(
\frac{(u^\varepsilon)^2}{2}-\varepsilon\partial_xu^\varepsilon-R^\varepsilon
\big)\partial_x\phi
\Big)\,\mathrm{d}x\,\mathrm{d}t
+\int_{\mathbb{R}}u^\varepsilon_{\mathrm{in}}\phi(0)\,\mathrm{d}x=0
\end{split}
\end{equation}
 for all $\phi\in \mathrm{C}^1_\mathrm{c}([0,\infty)\times\mathbb{R})$ and where 
\begin{equation*}
R^\varepsilon:=\dfrac{\dx u^\varepsilon}{\sqrt{\vert \dx u^\varepsilon \vert^2 + \varepsilon^2}}  \,.
\end{equation*}   
The function $u^\eps$ is of the regularity 
\begin{subequations}
\label{reg-ueps}
\begin{align}
\label{reg-ueps1}
u^\eps&\in L^\infty(0,T;H^2(\R))\phantom{\cap L^2(0,T;H^2(\R))}\quad\text{such that}\\
\label{reg-ueps2}
\partial_tu^\eps&\in L^\infty(0,T;L^2(\R))\cap L^2(0,T;H^1(\R))\,,
\end{align}
\end{subequations}
and it also satisfies the following integral far-field relations for all $T\in(0,\infty)$:
\begin{subequations}
\label{int-far-field}
\begin{eqnarray}
\label{int-far-field1}
\int_0^T\tfrac{1}{3}u^\varepsilon(\pm r,t)^3\,\mathrm{d}t&\to&0\quad\text{as }r\to\infty\,,\\
\label{int-far-field2}
\int_0^Tu^\varepsilon(\pm r,t)\varepsilon \partial_xu^\varepsilon(\pm r,t)\,\mathrm{d}t&\to&0\quad\text{as }r\to\infty\,,\\
\label{int-far-field3}
\int_0^Tu^\varepsilon(\pm r,t)\frac{\partial_xu^\varepsilon(\pm r,t)}{\sqrt{|\partial_xu^\varepsilon(\pm r,t)|^2+\varepsilon^2}}\,\mathrm{d}t&\to&0\quad\text{as }r\to\infty\,.
\end{eqnarray}
\end{subequations}
\end{thm}
\EEE
%
\subsection{Uniform a priori estimates}
\label{sec:apriori}
%
In the following we deduce uniform a priori estimates for the sequence of solutions $(u^\varepsilon)_\varepsilon$ of \eqref{ibvp-approx} obtained by Thm.\ \ref{Thm-Ex-Approx-Sol}. 
In a first step, in Section \ref{sec:first-est}, Proposition \ref{EDB-UniBds}, we deduce uniform estimates as a direct consequence of an energy-dissipation balance. In a second step, in Section \ref{sec:impro-est}, we will obtain further, improved estimates by testing \eqref{eq:b-p-app} with suitably tailored test functions.
%
\subsubsection{First uniform a priori estimates from the energy-dissipation balance}
\label{sec:first-est}
%
In a first step, we get the following direct estimates: 
\begin{proposition}[Energy-dissipation balance and first uniform a priori estimates]
\label{EDB-UniBds}
Let the assumptions of Thm.\ \ref{Thm-Ex-Approx-Sol} hold true. 
For every $\varepsilon>0$ let $u^\varepsilon$ be a solution of the approximate Cauchy problem \eqref{ibvp-approx} obtained in Thm.\ \ref{Thm-Ex-Approx-Sol}. Then, the following statements hold true: 
\begin{enumerate}
\item 
For all $\varepsilon>0$ and all $T\in(0,\infty)$ the following energy-dissipation balance is satisfied
\begin{equation}
\label{en-dissip-eps}
\begin{split}
&\frac{1}{2}\|u^\varepsilon(T)\|_{L^2(\mathbb{R})}^2
+\varepsilon\int_0^T\|\partial_xu^\varepsilon\|^2_{L^2(\mathbb{R})}\,\mathrm{d}t
+\int_0^T\int_{\mathbb{R}}\frac{|\partial_x u^\varepsilon|^2}{\sqrt{|\partial_xu^\varepsilon|^2+\varepsilon^2}}\,\mathrm{d}x\,\mathrm{d}t
=\frac{1}{2}\|u^\varepsilon_\mathrm{in}\|_{L^2(\mathbb{R})}^2\,.
\end{split}
\end{equation}
\item
Further assume that the initial data are unformly bounded in $L^2(\mathbb{R}),$ i.e., there is a constant $c>0$ such that for all $\varepsilon>0$ it is in particular 
\begin{equation}
\label{unif-bd-ini}
\frac{1}{2}\|u^\varepsilon_{\mathrm{in}}\|_{L^2(\mathbb{R})}^2\leq c\,.
\end{equation}
Then, there is a constant $C>0$ such that the following bounds hold true for all $\varepsilon>0$ and for all $T>0$:
\begin{subequations}
\label{unif-bd}
\begin{eqnarray}
\label{unif-bd-Linfty-L2}
\|u^\varepsilon(T)\|_{L^2(\mathbb{R})}&\leq&C
\,,\\
\label{unif-bd-viscosity}
\sqrt{\varepsilon}\|\partial_xu^\varepsilon\|_{L^2((0,T)\times\mathbb{R})}&\leq&C\,.
\end{eqnarray}
Additionally, for all $\varepsilon>0,$ all $T>0,$ and all $B_r:=(-r,r)$ with $r>0$ also the following bounds hold true: 
\begin{eqnarray}
\label{unif-bd-Reps-Linfty}
\|R^\varepsilon\|_{L^\infty((0,T)\times B_r)}&\leq& 1
\;\text{ for }
R^\varepsilon:=\frac{\partial_x u^\varepsilon}{\sqrt{|\partial_xu^\varepsilon|^2+\varepsilon^2}}
\,,\\
\label{unif-bd-BVloc}
\|\partial_xu^\varepsilon\|_{L^1((0,T)\times B_r)}&\leq& c+2rT\varepsilon
\,,\\
\label{unif-bd-time-der}
\|\partial_tu^\varepsilon\|_{L^2(0,T;(W_0^{1,\infty}(\mathbb{R}))')}&\leq&C(T)
\,.
\end{eqnarray}
In addition, there are subsequences $(T_\varepsilon)_\varepsilon,$ $(r_\varepsilon)_\varepsilon$ with $T_\varepsilon\nearrow\infty$ and $r_\varepsilon\nearrow\infty$ monotonically as $\varepsilon\to0,$ such that
\begin{equation}
\label{unif-bd-BV}
\|\partial_xu^\varepsilon\|_{L^1((0,T_\varepsilon)\times B_{r_\varepsilon})}\leq c+1
\,.\\
\end{equation}
\end{subequations}
\end{enumerate}
\end{proposition}
\begin{proof}
{\bf Proof of energy-dissipation estimate \eqref{en-dissip-eps}: } 
Multiplying equation \eqref{eq:b-p-app} with $u^\varepsilon,$ integrating over $Q_r^T:=(0,T)\times B_r,$ and performing an integration by parts results in 
\begin{equation*}
\begin{split}
0&=\int_0^T\int_{B_r}\Big(
(\partial_t u^\varepsilon)u^\varepsilon
+(u^\varepsilon\partial_xu^\varepsilon)u^\varepsilon
+\partial_x(-\varepsilon\partial_xu^\varepsilon-R^\varepsilon)u^\varepsilon
\Big)\,\mathrm{d}x\,\mathrm{d}t\\
&=\frac{1}{2}\|u^\varepsilon(T)\|_{L^2(B_r)}^2
-\frac{1}{2}\|u^\varepsilon(0)\|_{L^2(B_r)}^2+\int_0^T\int_{B_r}\Big(
\varepsilon(\partial_xu^\varepsilon)^2+R^\varepsilon\partial_xu^\varepsilon
\Big)\,\mathrm{d}x\,\mathrm{d}t\\
&\quad
+\int_0^T\frac{1}{3}\big(
(u^\varepsilon(r,t))^3-(u^\varepsilon(-r,t))^3
\big)\,\mathrm{d}t\\
&\quad
-\int_0^T\varepsilon\Big(u^\varepsilon(r,t)\partial_xu^\varepsilon(r,t)
-u^\varepsilon(-r,t)\partial_xu^\varepsilon(-r,t)\Big)\,\mathrm{d}t\\
&\quad
-\int_0^T\varepsilon\Big(u^\varepsilon(r,t)R^\varepsilon(r,t)
-u^\varepsilon(-r,t)R^\varepsilon(-r,t)\Big)\,\mathrm{d}t\,.
\end{split}
\end{equation*}
Letting $r\to\infty$ and invoking the integral far-field boundary conditions \eqref{int-far-field} in the last three lines above gives \eqref{en-dissip-eps}. 
\\[1ex]
{\bf Proof of a priori estimates \eqref{unif-bd-Linfty-L2} and \eqref{unif-bd-viscosity}: } 
Using the uniform bound \eqref{unif-bd-ini} for the initial data in the energy-dissipation balance \eqref{en-dissip-eps} directly shows \eqref{unif-bd-Linfty-L2} and \eqref{unif-bd-viscosity}.
\\[1ex]
{\bf Proof of \eqref{unif-bd-Reps-Linfty}: }From \eqref{unif-bd-viscosity} we deduce for all $\varepsilon>0$ fixed that 
$|\partial_x u^\varepsilon|<\infty$ $\mathcal{L}^2$-a.e.\ in $(0,T)\times B_r$ for any $T>0$ and any $r>0$. 
Therefore, we can conclude that 
\begin{equation*}
0\leq |R^\varepsilon|=\frac{|\partial_xu^\varepsilon|}{\sqrt{|\partial_xu^\varepsilon|^2+\varepsilon^2}}\leq1\quad\text{ a.e.\ in }(0,\infty)\times\mathbb{R}\,,
\end{equation*}
which is \eqref{unif-bd-Reps-Linfty}.
\\[1ex]
{\bf Proof of \eqref{unif-bd-BVloc}: } Using once more \eqref{unif-bd-ini} in \eqref{en-dissip-eps} provides for any $r>0$
\begin{equation}
\label{unif-bd-BVloc1}
c\geq \int_0^T\int_{\mathbb{R}}\frac{|\partial_x u^\varepsilon|^2}{\sqrt{|\partial_xu^\varepsilon|^2+\varepsilon^2}}\,\mathrm{d}x\,\mathrm{d}t
\geq\int_0^T\int_{B_r}\frac{|\partial_x u^\varepsilon|^2}{\sqrt{|\partial_xu^\varepsilon|^2+\varepsilon^2}}\,\mathrm{d}x\,\mathrm{d}t\,.
\end{equation} 
For any $y\in{\mathbb{R}}^d$ we further calculate that 
\begin{equation}
\label{unif-bd-BVloc2}
\begin{split}
\frac{|y|^2}{\sqrt{|y|^2+\varepsilon^2}}
=&\sqrt{|y|^2+\varepsilon^2}+\frac{|y|^2}{\sqrt{|y|^2+\varepsilon^2}}
-\frac{(\sqrt{|y|^2+\varepsilon^2})^2}{\sqrt{|y|^2+\varepsilon^2}}\\
=&\sqrt{|y|^2+\varepsilon^2}-\frac{\varepsilon^2}{\sqrt{|y|^2+\varepsilon^2}}\\
\geq&\sqrt{|y|^2+\varepsilon^2}-\varepsilon\,.
\end{split}
\end{equation}
Thus, inserting \eqref{unif-bd-BVloc2} into \eqref{unif-bd-BVloc1} results in 
\begin{equation*}
c\geq\int_0^T\int_{B_r}|\partial_xu^\varepsilon|\,\mathrm{d}x\,\mathrm{d}t-2rT\varepsilon\,,
\end{equation*}
which is \eqref{unif-bd-BVloc}.
\\[1ex]
{\bf Proof of \eqref{unif-bd-time-der}: }We test \eqref{eq:b-p-app} with a suitable test function $\phi,$ integrate over 
$(0,T)\times\mathbb{R},$ and perform an integration by parts in space. In particular, assuming that 
$\phi(t,\cdot)$ has compact support for a.a.\ $t\in(0,T)$ thus gives 
\begin{equation}
\label{unif-bd-time-der1}
\int_0^T\int_{\mathbb{R}}
\Big(
\partial_tu^\varepsilon\phi
+\big(\varepsilon\partial_xu^\varepsilon+R^\varepsilon-\frac{(u^\varepsilon)^2}{2}\big)
\partial_x\phi
\Big)\,\mathrm{d}x\,\mathrm{d}t=0\,.
\end{equation}
Estimates \eqref{unif-bd-Linfty-L2}, \eqref{unif-bd-viscosity}, and \eqref{unif-bd-Reps-Linfty} ensure that 
$\partial_x u^\varepsilon\in L^2((0,T)\times\mathbb{R}),$ $R^\varepsilon\in L^\infty((0,T)\times\mathbb{R}),$ and $(u^\varepsilon)^2\in L^\infty(0,\infty;L^1(\mathbb{R}))$. Hence, the term 
$\big(\varepsilon\partial_xu^\varepsilon+R^\varepsilon-\frac{(u^\varepsilon)^2}{2}\big)
\partial_x\phi\in L^1((0,T)\times\mathbb{R})$  for any $\partial_x\phi\in L^2(0,T;L^\infty(\mathbb{R}))$ with compact support, i.e., in particular for $\phi\in L^2(0,T;W^{1,\infty}_0(\mathbb{R}))$. Using once more the uniform bounds \eqref{unif-bd-Linfty-L2}, \eqref{unif-bd-viscosity}, and \eqref{unif-bd-Reps-Linfty}, we now conclude \eqref{unif-bd-time-der}  by comparison in \eqref{unif-bd-time-der1}.
\\[1ex]
{\bf Proof of \eqref{unif-bd-BV}: }This estimate directly follows from \eqref{unif-bd-BVloc} by suitably choosing two monotonically increasing  subsequences satisfying $T_\varepsilon\nearrow\infty,$ $r_\varepsilon\nearrow\infty,$ and $T_\varepsilon r_\varepsilon\leq1/(2\varepsilon)$.  
 \end{proof}
 %
 \subsubsection{Further, improved estimates}
 \label{sec:impro-est}
 %
 Next, we deduce further uniform estimates by testing \eqref{eq:b-p-app} by suitably tailored test functions. For this, for instance, one would like to test \eqref{eq:b-p-app} for $\varepsilon>0$ fixed  with the characteristic  function of $\mathcal{L}^2$-measurable sets $A^\varepsilon\subset(0,\infty)\times\mathbb{R}$, such as $A^\varepsilon:=\{\alpha^\varepsilon>\beta\}$ with some $\mathcal{L}^2$-measurable functions $\alpha^\varepsilon:(0,\infty)\times\mathbb{R}\to\mathbb{R}$ and a constant $\beta\in\mathbb{R}$. In particular, $\alpha^\varepsilon$ here serves as a placeholder for the smooth, approximate solutions $u^\varepsilon$ or its derivatives $\partial_x u^\varepsilon,$ $\partial_t u^\varepsilon$.  
However, the characteristic function of such a set $A^\varepsilon$ is not an admissible test function for \eqref{eq:b-p-app}. Instead, we will use a suitable regularization of it, which can be constructed in terms of a regularization of the Heaviside function by means of Lemma \ref{Help-Indicator}. We point out that this result can be found in \cite[Thm.\ 6.3.2]{dafermosHyperbolicConservationLaws2000}. 
However, for the readers' convience we give the statement and proof of the result here below in the way we use it to conclude the uniform estimates in Proposition \ref{UniBDsImpro} lateron.
\begin{lm}
\label{Help-Indicator}
Let $\rho\in \mathrm{C}^\infty_0(\mathbb{R})$ be a nonnegative, even function such that 
\begin{equation}
\label{ass-rho}
\int_{\mathbb{R}}\rho(y)\,\mathrm{d}y=1
\quad\text{and}\quad
\rho(y)=0\;\text{ for }|y|\geq1\,.
\end{equation}
For any $\sigma\in(0,1)$ define the following function 
\begin{equation}
\label{def-Hsigma}
H^\sigma:\mathbb{R}\to(0,1],\; 
H^\sigma(y):=\int_{-\infty}^y\frac{1}{\sigma}\rho\big(\frac{y'}{\sigma}\big)\,\mathrm{d}y'\,.
\end{equation}
Then, $(H^\sigma)_{\sigma>0}$ approximates the Heaviside function pointwise for all $y\in\mathbb{R}\backslash\{0\}$; more precisely, as $\sigma\to0,$ there holds
\begin{equation}
\label{prop-Hsigma}
H^\sigma(y)\to\left\{\begin{array}{cl}
0&\text{ if }y<0,\\
1/2&\text{ if }y=0,\\
1&\text{ if }y>0.
\end{array}\right.
\end{equation}
Let  $\alpha\in L^p((0,\infty)\times\mathbb{R})$
such that also for all $t\in[0,\infty)$ it is $\alpha(t,\cdot)\in L^p(\mathbb{R})$ for some $p\in[1,\infty)$. Further, let $\beta\in\mathbb{R}$ be a constant and set 
\begin{equation}
\label{def-chisigma}
\chi_{\{\alpha>\beta\}}^\sigma:=H^\sigma(\alpha-\beta)\,.
\end{equation}
Then, as $\sigma\to0$, there holds for all $t\in[0,\infty)$  
\begin{equation}
\label{prop-chi1}
(\alpha(t)-\beta)\chi_{\{\alpha>\beta\}}^\sigma(t)\to
(\alpha(t)-\beta)^+\;\text{ pointwise a.e.\ in }
\mathbb{R}\,,
\end{equation}
where $(\cdot)^+:=\max\{\cdot,0\},$ 
as well as 
\begin{subequations}
\label{prop-chi2}
\begin{eqnarray}
\label{prop-chi2-1}
(\alpha(t)-\beta)\chi_{\{\alpha>\beta\}}^\sigma&\to&(\alpha(t)-\beta)^+\;\text{ in }
L^p(\mathbb{R})\,,
\\
\label{prop-chi2-2}
(\alpha-\beta)\chi_{\{\alpha>\beta\}}^\sigma&\to&(\alpha-\beta)^+\;\quad\text{ in }
L^p((0,\infty)\times\mathbb{R})\,.
\end{eqnarray}
\end{subequations}
Moreover, there holds 
\begin{equation}
\label{prop-chi3}
\frac{\alpha-\beta}{\sigma}\rho\big(\frac{\alpha-\beta}{\sigma}\big)\to0\quad\text{ pointwise a.e.\ in }
(0,\infty)\times\mathbb{R}\quad\text{ as }\sigma\to0\,. 
\end{equation}
Furthermore, let $f\in L^1((0,\infty)\times\mathbb{R})$.  Then, there is a constant $c>0$ such that 
\begin{equation}
\label{prop-chi4}
\left|\int_0^\infty\int_{\mathbb{R}}
\frac{\alpha}{\sigma}\rho\big(\frac{\alpha}{\sigma}\big)f
\,\mathrm{d}x\,\mathrm{d}t\right|
\leq c\int_{\{|\alpha-\beta|<\sigma\}\backslash\{|\alpha-\beta|=0\}}
\hspace*{-5ex}|f|\,\mathrm{d}x\,\mathrm{d}t\to0
\quad\text{ as }\sigma\to0\,.
\end{equation}
\end{lm} 
\begin{proof}
{\bf Proof of \eqref{prop-Hsigma}: }With $z:=y'/\sigma,$ the integral is transformed as follows
\begin{equation}
H^\sigma(y)=\int_{-\infty}^y\frac{1}{\sigma}\rho\big(\frac{y'}{\sigma}\big)\,\mathrm{d}y'
=\int_{-\infty}^{y/\sigma}\rho(z)\,\mathrm{d}z\,.
\end{equation}
Accordingly, for $y<0,$ we find that $H^\sigma(y)\to\int_{-\infty}^{-\infty}\rho(z)\,\mathrm{d}z=0$ and for $y>0$ we see that $H^\sigma(y)\to\int_{-\infty}^{\infty}\rho(z)\,\mathrm{d}z=1$ by \eqref{ass-rho}. 
Instead, since $\rho$ is nonnegative, even, and satisfies \eqref{ass-rho}, we conclude that $H^\sigma(0)=1/2$ for all $\sigma>0,$ which yields \eqref{prop-Hsigma}. 
\\[1ex]
{\bf Proof of \eqref{prop-chi1}: }In view of \eqref{prop-Hsigma} we now deduce that
\begin{equation*}
\begin{split}
&(\alpha(t)-\beta)\chi_{\{\alpha>\beta\}}^\sigma(t)\\
&=(\alpha(t)-\beta)H^\sigma(\alpha(t)-\beta)
\to\left\{\begin{array}{cl}
0&\text{ if }(\alpha(t)-\beta)<0,\\
0&\text{ if }(\alpha(t)-\beta)=0,\\
(\alpha(t)-\beta)&\text{ if }(\alpha-\beta)>0
\end{array}\right\}
=(\alpha(t)-\beta)^+\,,
\end{split}
\end{equation*}
which proves \eqref{prop-chi1}.
\\[1ex]
{\bf Proof of \eqref{prop-chi2}: }For brevity we set $y:=(\alpha-\beta)$ and conclude by the dominated convergence theorem that 
\begin{equation*}
\begin{split}
&\|y\chi_{\{y>0\}}^\sigma-(y)^+\|_{L^p((0,\infty)\times\R)}^p\\
&=\int_{\{y<0\}}\Big|y\int_{-\infty}^{y/\sigma}\rho(z)\,\mathrm{d}z\Big|^p\,\mathrm{d}x\,\mathrm{d}t
+\int_{\{y>0\}}\Big|y\Big(1-\int_{-\infty}^{y/\sigma}\rho(z)\,\mathrm{d}z\Big)\Big|^p\,\mathrm{d}x\,\mathrm{d}t\\
&\quad\longrightarrow 0\,,
\end{split}
\end{equation*}
where we used that each of the integrands tends to $0$ pointwise a.e.\ on its domain of integration by \eqref{prop-chi1} and that the function $y$ serves as a $p$-integrable majorant for both of them. 
\\[1ex]
{\bf Proof of \eqref{prop-chi3}: }Again with $y:=\alpha-\beta$ we note that
\begin{equation}
\label{prop-chi3-1}
\frac{y}{\sigma}\rho\big(\frac{y}{\sigma}\big)\in\left\{\begin{array}{cl}
\{0\}&\text{ if }|y|=0,\\
(0,c)&\text{ if }|y|<\sigma,\\
\{0\},&\text{ if }|y|\geq\sigma\\
\end{array}\right.
\end{equation}
with some constant $c>0$. In fact, the existence of the bound $c>0$ for $|y|<\sigma$ 
is due to $|y|/\sigma<1$ and $\rho\in \mathrm{C}^\infty_0(\mathbb{R})$. 
\par 
Since, for any $t\in[0,\infty)$ the function $y(t,\cdot)\in L^p(\mathbb{R})$ by assumption, we have that for a.e.\ $x\in\mathbb{R}$ with $y(t,x)\not=0$ there also holds $|y(t,x)|<\infty$. Hence, there is some $\sigma_x^t>0$ such that $\frac{|y(t,x)|}{\sigma_x^t}\geq1$ and, accordingly 
$\rho\big(\frac{|y(t,x)|}{\sigma_x^t}\big)=0$. Thus, 
\begin{equation}
\frac{|y(t,x)|}{\sigma}\rho\big(\frac{y(t,x)}{\sigma}\big)=0\quad\text{ also for any }\sigma\leq\sigma_x^t\,,
\end{equation}
which proves pointwise convergence for a.e.\ $x$ with $|y(t,x)|>0$. Using that 
$\frac{y(t,x)}{\sigma}\rho\big(\frac{y(t,x)}{\sigma}\big)=0$ for any $x\in\mathbb{R}$ with $y(t,x)=0$ 
completes the proof of convergence property \eqref{prop-chi3}. 
\\[1ex]
{\bf Proof of \eqref{prop-chi4}: }The existence of the constant $c>0$ has been 
justified in \eqref{prop-chi3-1}. Also by \eqref{prop-chi3-1} we further deduce the estimate 
\begin{equation}
\label{prop-chi4-1}
\left|\int_0^\infty\int_{\mathbb{R}}
y\rho(y)f
\,\mathrm{d}x\,\mathrm{d}t\right|
\leq c\int_{\{|y|<\sigma\}\backslash\{|y|=0\}}
\hspace*{-5ex}|f|\,\mathrm{d}x\,\mathrm{d}t\,.
\end{equation}
Moreover, for any $\sigma_2<\sigma_1$ it is $\{|y|<\sigma_2\}\subset\{|y|<\sigma_1\}$ and hence
\begin{equation}
\label{prop-chi4-2}
\lim_{\sigma\to0}\{|y|<\sigma\}=\bigcap_{\sigma>0}\{|y|<\sigma\}=\{|y|=0\}\,.
\end{equation} 
Using \eqref{prop-chi4-2} in \eqref{prop-chi4-1} concludes the proof of \eqref{prop-chi4} for $f\in L^1((0,\infty)\times\mathbb{R})$. 
\end{proof} 
Based on Lemma \ref{Help-Indicator} we deduce the following uniform estimates
\begin{proposition}[Improved estimates and Oleinik's entropy condition]
\label{UniBDsImpro}
Let the assumptions of Proposition \ref{EDB-UniBds} hold true. 
\begin{enumerate}
\item
Further assume that $u_\mathrm{in},u_\mathrm{in}^\varepsilon\in L^\infty(\mathbb{R})\cap L^1(\mathbb{R})$ such that for all $\varepsilon>0$ 
\begin{equation}
\label{ass-uin}
\underline{u}\leq\mathrm{ess\,inf}\{u_\mathrm{in},u_\mathrm{in}^\varepsilon\}\leq\mathrm{ess\,sup}\{u_\mathrm{in},u_\mathrm{in}^\varepsilon\}\leq\bar u
\end{equation} 
uniformly in $\varepsilon$ for two constants $\underline{u}\leq\bar u\in\mathbb{R}$. Then, there also holds for all $\varepsilon>0$
\begin{equation}
\label{unif-bd-ueps-Linfty}
\|u^\varepsilon\|_{L^\infty(0,\infty;L^\infty(\mathbb{R}))}\leq\max\{|\underline{u}|,|\bar u|\}\,,
\end{equation}
as well as
\begin{equation}
\label{unif-bd-ueps-L1}
\|u^\varepsilon\|_{L^\infty(0,\infty;L^1(\mathbb{R}))}\leq \|u^\eps_{\mathrm{in}}\|_{L^1(\R)}\,.
\end{equation}
\item 
There also holds for all $\varepsilon>0$
\begin{equation}\label{unif-bd-ueps-BV}
    \norm{\dx u^\varepsilon}{L^\infty(0,\infty;L^1(\mathbb R))} \leq \norm{\dx u^\varepsilon_\mrm{in}}{L^1(\mathbb R)} 
    \,.
\end{equation}
\item 
The approximating solutions $(u^\varepsilon)_\varepsilon$ also satisfy Oleinik's entropy condition uniformly  for all $\varepsilon>0,$ i.e., for all $t>0$ and for all $x\in\R$ there holds 
\begin{equation}
\label{unif-bd-ueps-Oleinik}
\begin{split}
\dx u^\varepsilon (t,x) < \dfrac{1}{t}\,.
\end{split}    
\end{equation}
\item Assume in addition, that also $\partial_tu_\mathrm{in}^\varepsilon\in  L^1(\mathbb{R})$ such that for all $\varepsilon>0$ 
\begin{equation}
\label{ass-partialt-uin}
\|\partial_tu_\mathrm{in}^\varepsilon\|_{L^1(\mathbb{R})}\leq C
\end{equation} 
uniformly in $\varepsilon$ for a constant $C>0,$ where 
\begin{equation}
\partial_tu^\varepsilon_\mathrm{in}
:=\partial_x\Big(-\frac{(u^\varepsilon_{\mathrm{in}})^2}{2}
+\frac{(u^\varepsilon_{\mathrm{in}})^2}{\sqrt{|u^\varepsilon_{\mathrm{in}}|^2+\varepsilon^2}}
+\varepsilon\partial_xu^\varepsilon_{\mathrm{in}}
\Big)
\end{equation}
 Then, there also holds for all $\varepsilon>0$
\begin{equation}
\label{unif-bd-partialt-ueps-LinftyL1}
\|\partial_tu^\varepsilon\|_{L^\infty(0,\infty;L^1(\mathbb{R}))}\leq C\,.
\end{equation} 
\item Under the assumptions of 1.\ there is a constant $C>0$ such that for all $\varepsilon>0$
\begin{equation}
\label{unif-bd-drifteps-LinftyL1}
\|\partial_x(u^\varepsilon)^2/2\|_{L^\infty(0,\infty;L^1(\mathbb{R}))}\leq C\,.
\end{equation}
\item Under the assumptions of 1.\ and 4.\ there is a constant $C>0$ such that for all $\varepsilon>0$
\begin{equation}
\label{unif-bd-visceps-LinftyL1}
\|\partial_x(R^\varepsilon
+\varepsilon\partial_xu^\varepsilon)\|_{L^\infty(0,\infty;L^1(\mathbb{R}))}
\leq C\,.
\end{equation}
\end{enumerate}
\end{proposition} 
\begin{proof}
{\bf Proof of 1.: }
In order to verify \eqref{unif-bd-ueps-Linfty} we first show that $u^\varepsilon\leq \bar u$ in $(0,\infty)\times\mathbb{R}$. For this, we consider the set $\{u^\varepsilon>\bar u\}$ and the regularization of its characteristic functions $\chi^\sigma_{\{u^\varepsilon>\bar u\}}$ as proposed in Lemma \ref{Help-Indicator}. We shift \eqref{eq:b-p-app} by $-\bar u$ and test the resulting equation by  
$\chi^\sigma_{\{u^\varepsilon>\bar u\}}$. Integrating over $(0,T)\times B_r$ and performing an integration by parts in space thus results in 
\begin{equation}
\label{unif-bd-ueps-Linfty1}
\begin{split}
0=&\int_0^T\int_{B_r}\partial_t\big((u^\varepsilon-\bar u)
\chi^\sigma_{\{u^\varepsilon>\bar u\}}\big)
\,\mathrm{d}x\,\mathrm{d}t
\\
&-\int_0^T\int_{B_r}
(u^\varepsilon-\bar u)\partial_t\chi^\sigma_{\{u^\varepsilon>\bar u\}}\big)
\,\mathrm{d}x\,\mathrm{d}t
\\
&+\int_0^T\int_{B_r}
\partial_xu^\varepsilon(u^\varepsilon-\bar u)\chi^\sigma_{\{u^\varepsilon>\bar u\}}
\,\mathrm{d}x\,\mathrm{d}t
\\
&+\int_0^T\int_{B_r}
\Big(\varepsilon\partial_x(u^\varepsilon-\bar u)
+\frac{\partial_x(u^\varepsilon-\bar u)}{\sqrt{|\partial_xu^\varepsilon|^2+\varepsilon^2}}\Big)
\partial_x\chi^\sigma_{\{u^\varepsilon>\bar u\}}
\,\mathrm{d}x\,\mathrm{d}t
\\
&-\int_0^T\Big[
\Big(\varepsilon\partial_x(u^\varepsilon-\bar u)
+\frac{\partial_x(u^\varepsilon-\bar u)}{\sqrt{|\partial_xu^\varepsilon|^2+\varepsilon^2}}\Big)
\chi^\sigma_{\{u^\varepsilon>\bar u\}}
\Big]_{-r}^r\,\mathrm{d}t\,,
\end{split}
\end{equation}
and we observe that that the last term tends to $0$ as $r\to\infty$ again by the integral far-field relations \eqref{int-far-field}. 
Moreover, for the second-last term in \eqref{unif-bd-ueps-Linfty1} we estimate that 
\begin{equation}
\label{unif-bd-ueps-Linfty2}
\begin{split}
&\int_0^T\int_{B_r}
\Big(\varepsilon\partial_x(u^\varepsilon-\bar u)
+\frac{\partial_x(u^\varepsilon-\bar u)}{\sqrt{|\partial_xu^\varepsilon|^2+\varepsilon^2}}\Big)
\partial_x\chi^\sigma_{\{u^\varepsilon>\bar u\}}
\,\mathrm{d}x\,\mathrm{d}t\\
&=
\int_0^T\int_{B_r}
\Big(\varepsilon\partial_x(u^\varepsilon-\bar u)
+\frac{\partial_x(u^\varepsilon-\bar u)}{\sqrt{|\partial_xu^\varepsilon|^2+\varepsilon^2}}\Big)
\partial_x(u^\varepsilon-\bar u)\frac{1}{\sigma}\rho\big(\frac{u^\varepsilon-\bar u}{\sigma}\big)
\,\mathrm{d}x\,\mathrm{d}t\\
&\geq0\,
\end{split}
\end{equation}
for all $\sigma>0$ and all $\varepsilon>0$ by the non-negativity of $\rho$ claimed in Lemma \ref{Help-Indicator}. 
\par 
Next, we show that the  second term on the right-hand side of \eqref{unif-bd-ueps-Linfty1} tends to $0$ as $\sigma\to0$. For this we  observe that
\begin{equation}
\label{unif-bd-ueps-Linfty3}
\begin{split}
&\left|\int_0^T\int_{B_r}
(u^\varepsilon-\bar u)\partial_t\chi^\sigma_{\{u^\varepsilon>\bar u\}}
\,\mathrm{d}x\,\mathrm{d}t\right|
\\
&\leq\int_0^T\int_{\mathbb{R}}\left|
(u^\varepsilon-\bar u)\partial_t(u^\varepsilon-\bar u)\frac{1}{\sigma}\rho\big(\frac{u^\varepsilon-\bar u}{\sigma}\big)\right|
\,\mathrm{d}x\,\mathrm{d}t
\\
&\leq\int_{\{|u^\varepsilon-\bar u|<\sigma\}}
\left|\partial_t(u^\varepsilon-\bar u)\frac{(u^\varepsilon-\bar u)}{\sigma}\rho\big(\frac{u^\varepsilon-\bar u}{\sigma}\big)\right|
\,\mathrm{d}x\,\mathrm{d}t
\\
&\leq c\int_{\{|u^\varepsilon-\bar u|<\sigma\}\backslash\{|u^\varepsilon-\bar u|=0\}}
|\partial_t(u^\varepsilon-\bar u)|\,\mathrm{d}x\,\mathrm{d}t
\\
&\to0%
\end{split}
\end{equation} 
as $\sigma\to0$ by property \eqref{prop-chi4} from Lemma \ref{Help-Indicator}.
\par 
Thus, inserting \eqref{unif-bd-ueps-Linfty2} and \eqref{unif-bd-ueps-Linfty3} in \eqref{unif-bd-ueps-Linfty1} and letting $r\to\infty$ and $\sigma\to0,$ results in the following estimate for all $T>0$
\begin{equation}
\label{unif-bd-ueps-Linfty4}
\begin{split}
&\int_\mathbb{R}\Big((u^\varepsilon(T)-\bar u)^+-(u^\varepsilon_\mathrm{in}-\bar u)^+\Big)\,\mathrm{d}x
\\
&=\lim_{\sigma\to0}\int_{\mathbb{R}}(u^\varepsilon(T)-\bar u)\chi^\sigma_{\{u^\varepsilon>\bar u\}}(T)-(u^\varepsilon_\mathrm{in}-\bar u)^+\chi^\sigma_{\{u^\varepsilon_{\mathrm{in}}>\bar u\}}
\,\mathrm{d}x\,\mathrm{d}t
\\
&=\lim_{\sigma\to0}\int_0^T\int_{\mathbb{R}}\partial_t\big((u^\varepsilon-\bar u)
\chi^\sigma_{\{u^\varepsilon>\bar u\}}\big)
\,\mathrm{d}x\,\mathrm{d}t\\
&\leq 
\lim_{\sigma\to0}\Big(-\int_0^T\int_{\mathbb{R}}
\partial_xu^\varepsilon(u^\varepsilon-\bar u)\chi_{\{u^\varepsilon>\bar u\}}^\sigma
\,\mathrm{d}x\,\mathrm{d}t
\\
&\qquad\qquad
+c\int_{\{|u^\varepsilon-\bar u|<\sigma\}\backslash\{|u^\varepsilon-\bar u|=0\}}
|\partial_t(u^\varepsilon-\bar u)|\,\mathrm{d}x\,\mathrm{d}t
\Big)\\
&=
-\int_0^T\int_{\mathbb{R}}
\partial_xu^\varepsilon(u^\varepsilon-\bar u)\chi_{\{u^\varepsilon>\bar u\}}
\,\mathrm{d}x\,\mathrm{d}t\\
&\leq 
\|\partial_xu^\varepsilon\|_{L^\infty((0,\infty)\times\mathbb{R})}
\int_0^T\int_{\mathbb{R}}
(u^\varepsilon-\bar u)\chi_{\{u^\varepsilon>\bar u\}}
\,\mathrm{d}x\,\mathrm{d}t\\
&=\|\partial_xu^\varepsilon\|_{L^\infty((0,\infty)\times\mathbb{R})}
\int_0^T\int_{\mathbb{R}}
(u^\varepsilon-\bar u)^+
\,\mathrm{d}x\,\mathrm{d}t
\,,
\end{split}
\end{equation} 
where we have used the convergence properties \eqref{prop-chi2-1}, \eqref{prop-chi2-2}, and \eqref{prop-chi4} from Lemma \ref{Help-Indicator}. We also note that, by regularity properties 
\eqref{reg-ueps1}, we indeed have that $\partial_x u^\varepsilon\in L^\infty((0,\infty)\times\mathbb{R})$ for every $\varepsilon>0$. 
\par
Further observing that $\|(u^\varepsilon(\cdot)-\bar u)^+\|_{L^1(\mathbb{R})}:[0,\infty)\to[0,\infty)$ is continuous by the regularity properties \eqref{reg-ueps}, the integral version of Gr\"onwall's inequality now allows us to conclude from \eqref{unif-bd-ueps-Linfty4} that
\begin{equation}
\label{unif-bd-ueps-L10}
\|(u^\varepsilon(T)-\bar u)^+\|_{L^1(\mathbb{R})}
\leq 
\|(u^\varepsilon_\mathrm{in}-\bar u)^+\|_{L^1(\mathbb{R})}\exp\left(
T\|\partial_xu^\varepsilon\|_{L^\infty((0,\infty)\times\mathbb{R})}
\right)=0
\end{equation}
for all $T>0$ and all $\varepsilon>0,$ where the right-hand side in \eqref{unif-bd-ueps-L10} is $0$ thanks to assumption \eqref{ass-uin}. Hence, we conclude for every $T\geq0,$ $\varepsilon>0$ 
that $u^\varepsilon(T)\leq \bar u$ a.e.\ in $\mathbb{R}$. By the continuity of $u^\varepsilon(T):\mathbb{R}\to\mathbb{R}$ for every $T>0$ fixed this estimate is in fact true everywhere in $\mathbb{R},$ which proves the assertion. 
\par 
Next, it has to be shown that also $u^\varepsilon\geq\underline{u}$ in $(0,\infty)\times\mathbb{R}$ holds true. 
For this, we consider $(u^\varepsilon-\underline{u})_-:=-\min\{u^\varepsilon-\underline{u},0\}$ and, correspondingly, the set $\{u^\varepsilon<\underline{u}\}=\{-\underline{u}<-u^\varepsilon\}$.  Following the lines of the above proof, \eqref{eq:b-p-app} is shifted by $-\underline{u}$ and the sign is changed by a multiplication by $-1$. Now, again Lemma \ref{Help-Indicator} can be applied to $(-u^\varepsilon-(-\underline{u}))\chi_{\{-u^\varepsilon>\underline{u}\}}$ in order to deduce an estimate analogous to \eqref{unif-bd-ueps-Linfty4} for $(u^\varepsilon-\underline{u})_-$. Invoking again Gr\"onwall's inequality and assumption \eqref{ass-uin} as well as the continuity of $u^\varepsilon$ ultimately results in the desired estimate. Alltogether we have thus verified \eqref{unif-bd-ueps-Linfty}. 
\par 
Now we verify the $L^1$-bound \eqref{unif-bd-ueps-L1}. For this we notice that integrating 
equation \eqref{eq:b-p-app} over $(0,t)\times\R$ and performing an integration by parts gives
\begin{equation}
\label{unif-bd-ueps-L11}
\int_{\R} u^\eps(t)\,\mathrm{d}x=\int_{\R} u^\eps_{\mathrm{in}}\,\mathrm{d}x\,.
\end{equation} 
Moreover, multiplying \eqref{eq:b-p-app} with the regularized version $\chi_{\{u^\eps>0\}}^\sigma$ 
of the characteristic function $\chi_{\{u^\eps>0\}}$ and performing similar arguments as above, leads to 
\begin{equation}
\frac{\mathrm{d}}{\mathrm{d}t}\int_{\R}u^\eps\chi_{\{u^\eps>0\}}\,\mathrm{d}x\leq0\,.
\end{equation}
Integrating the above inequality in time yields
\begin{equation}
\label{unif-bd-ueps-L12}
\int_{\R}(u^\eps(t))^+\,\mathrm{d}x\leq\int_{\R}(u^\eps_{\mathrm{in}})^+\,\mathrm{d}x\,.
\end{equation}
Consequently one obtains from \eqref{unif-bd-ueps-L11} and \eqref{unif-bd-ueps-L12} that also 
\begin{equation}
\label{unif-bd-ueps-L13}
\int_{\R}(u^\eps(t))^-\,\mathrm{d}x\leq\int_{\R}(u^\eps_{\mathrm{in}})^-\,\mathrm{d}x\,,
\end{equation}
which allows us to conclude \eqref{unif-bd-ueps-L1}.
\par 
\noindent
{\bf Proof of 2.: }Applying $ \dx $ to \eqref{eq:b-p-app} yields
\begin{equation}
\label{unest:009}
    \dt \dx u^\varepsilon + 
    \partial_{xx}\frac{(u^{\varepsilon})^2}{2}
    - \varepsilon^2 \dx ( \dfrac{\partial_{xx} u^\varepsilon}{(\vert \dx u^\varepsilon\vert^2 + \varepsilon^2)^{3/2}}) - \varepsilon \partial_{xx} (\dx u^\varepsilon) = 0.
\end{equation}
Let $ \chi^\sigma_{\lbrace\dx u^\varepsilon > 0\rbrace} $ be the regularized characteristic function of the set $\{\dx u^\varepsilon > 0\}$ as introduced in Lemma \ref{Help-Indicator}.  
As for Item 1, we multiply \eqref{unest:009} with $ \chi^\sigma_{\lbrace \dx u^\varepsilon > 0\rbrace}, $ integrate over $(0,T)\times B_r,$ and perform an integration by parts in space to find 
\begin{equation}
\label{unif-bd-ueps-BV1}
\begin{split}
0=&\int_0^T\int_{B_r} \dt \big(\dx u^\varepsilon\chi^\sigma_{\lbrace \dx u^\varepsilon > 0\rbrace}\big) \,\mathrm{d}x\,\mathrm{d}t\,
-\int_0^T\int_{B_r}\dx u^\varepsilon\dt\chi^\sigma_{\lbrace \dx u^\varepsilon > 0\rbrace} \,\mathrm{d}x\,\mathrm{d}t\\
&-\int_0^T\int_{B_r} \frac{\partial_x(u^{\varepsilon})^2}{2}\partial_x\chi^\sigma_{\lbrace \dx u^\varepsilon > 0\rbrace}\,\mathrm{d}x\,\mathrm{d}t
+\int_0^T\Big[\frac{\partial_x(u^{\varepsilon})^2}{2}\chi^\sigma_{\lbrace \dx u^\varepsilon > 0\rbrace}\Big]_{-r}^r\,\mathrm{d}t\\
&+\int_0^T\int_{B_r}  \Big(\varepsilon^2  \dfrac{\partial_{xx} u^\varepsilon}{(\vert \dx u^\varepsilon\vert^2 + \varepsilon^2)^{3/2}} + \varepsilon \partial_{xx} u^\varepsilon\Big)\partial_x\chi^\sigma_{\lbrace \dx u^\varepsilon > 0\rbrace} \,\mathrm{d}x\,\mathrm{d}t\\
&-\int_0^T\Big[
\Big(\varepsilon^2  \dfrac{\partial_{xx} u^\varepsilon}{(\vert \dx u^\varepsilon\vert^2 + \varepsilon^2)^{3/2}} + \varepsilon \partial_{xx} u^\varepsilon\Big)\chi^\sigma_{\lbrace \dx u^\varepsilon > 0\rbrace}
\Big]_{-r}^r\,\mathrm{d}t\,.
\end{split}
\end{equation}
Arguing along the lines of Item 1 by making use of Lemma \ref{Help-Indicator} and of the regularity properties \eqref{reg-ueps} of the approximating solutions, we deduce for the different terms in \eqref{unif-bd-ueps-BV1} that
\begin{subequations}
\label{unif-bd-ueps-BV2}
\begin{align}
\nonumber
&\left|\int_0^T\int_{B_r}\dx u^\varepsilon\dt\chi^\sigma_{\lbrace \dx u^\varepsilon > 0\rbrace} \,\mathrm{d}x\,\mathrm{d}t\right|\\
&\leq c\int_{\{|\partial_xu^{\varepsilon}|<\sigma\}\backslash\{|\partial_xu^{\varepsilon}|=0\}}|\partial_t\partial_xu^{\varepsilon}|\,\mathrm{d}x\,\mathrm{d}t
\to0\quad
\text{ as }\sigma\to0\text{ for all }r>0\,,\\[1ex]
\nonumber
&\left|\int_0^T\int_{B_r} \frac{\partial_x(u^{\varepsilon})^2}{2}\partial_x\chi^\sigma_{\lbrace \dx u^\varepsilon > 0\rbrace}
\,\mathrm{d}x\,\mathrm{d}t\right|\\
&\leq c\int_{\{|\partial_xu^{\varepsilon}|<\sigma\}\backslash\{|\partial_xu^{\varepsilon}|=0\}}\hspace*{-5ex}
|u^{\varepsilon}\partial_{xx}\partial_xu^{\varepsilon}|\,\mathrm{d}x\,\mathrm{d}t
\to0\quad
\text{ as }\sigma\to0\text{ for all }r>0\,,\\[1ex]
&\int_0^T\Big[\frac{\partial_x(u^{\varepsilon})^2}{2}\chi^\sigma_{\lbrace \dx u^\varepsilon > 0\rbrace}\Big]_{-r}^r\,\mathrm{d}t
\to0\;\;\text{ as }r\to\infty\text{ for all }\sigma>0\,,\\[1ex]
&\int_0^T\int_{B_r}  \Big(\varepsilon^2  \dfrac{\partial_{xx} u^\varepsilon}{(\vert \dx u^\varepsilon\vert^2 + \varepsilon^2)^{3/2}} + \varepsilon \partial_{xx} u^\varepsilon\Big)\partial_x\chi^\sigma_{\lbrace \dx u^\varepsilon > 0\rbrace} \,\mathrm{d}x\,\mathrm{d}t
\geq0\\
\nonumber
&\hspace*{0.55\textwidth}\;\text{ for all }\sigma>0,\,r>0\,,\\[1ex]
&\int_0^T\Big[
\Big(\varepsilon^2  \dfrac{\partial_{xx} u^\varepsilon}{(\vert \dx u^\varepsilon\vert^2 + \varepsilon^2)^{3/2}} + \varepsilon \partial_{xx} u^\varepsilon\Big)\chi^\sigma_{\lbrace \dx u^\varepsilon > 0\rbrace}
\Big]_{-r}^r\,\mathrm{d}t\;\to0\;\text{ as }r\to\infty
\text{ for all }\sigma>0\,.
\end{align}
\end{subequations}
Putting the findings \eqref{unif-bd-ueps-BV2} together with \eqref{unif-bd-ueps-BV1} and letting first $r\to\infty$ and then $\sigma\to0$ 
results in the following estimate for all $T>0$
\begin{equation*}
\begin{split}
\int_{\mathbb{R}}\big(\partial_xu^\varepsilon(T)\big)^+-\big(\partial_xu^\varepsilon_{\mathrm{in}}\big)^+\,\mathrm{d}x
&
=\lim_{\sigma\to0}\int_0^T\frac{\mathrm{d}}{\mathrm{d}t}\int_{\mathbb{R}} \big(\dx u^\varepsilon\chi^\sigma_{\lbrace \dx u^\varepsilon > 0\rbrace}\big) \,\mathrm{d}x\,\mathrm{d}t\\
&\leq\lim_{\sigma\to0} c\int_{\{|\partial_xu^{\varepsilon}|<\sigma\}\backslash\{|\partial_xu^{\varepsilon}|=0\}}\hspace*{-5ex}
|u^{\varepsilon}\partial_{xx}\partial_xu^{\varepsilon}|\,\mathrm{d}x\,\mathrm{d}t\\
&\hspace*{0.2\textwidth}
+\lim_{\sigma\to0}c\int_{\{|\partial_xu^{\varepsilon}|<\sigma\}\backslash\{|\partial_xu^{\varepsilon}|=0\}}|\partial_t\partial_xu^{\varepsilon}|\,\mathrm{d}x\,\mathrm{d}t\\
&=0\,,
\end{split}
\end{equation*}
which yields for all $T>0$ 
\begin{equation}\label{unest:011}
    \int_{\mathbb{R}} (\dx u^\varepsilon(T))^+ \,\mathrm{d}x \leq \int_{\mathbb{R}} (\dx u^\varepsilon_\mrm{in})^+ \,\mathrm{d}x\,.
\end{equation}
In addition, thanks to \eqref{bc:far-field}, one has for all $T\geq0$
\begin{equation}
\label{unest:011a}
    \int_{\mathbb{R}} \dx u^\varepsilon(T) \,\mathrm{d}x = 0\,. 
\end{equation}
Therefore  one can conclude from \eqref{unest:011} and  \eqref{unest:011a} that
\begin{equation}\label{unest:012}
    \norm{\dx u^\varepsilon}{L^\infty(0,\infty;L^1(\mathbb R))} \leq 
    \norm{\dx u^\varepsilon_\mrm{in}}{L^1(\mathbb R)}\,.
\end{equation}
\par
\noindent
{\bf Proof of 3.\ Oleinik's entropy estimate: }
Consider the following ODE Cauchy problem 
\begin{subequations}
\begin{align}
    &\dfrac{\mathrm{d}}{\mathrm{d}t} Q + Q^2 = 0\quad\text{for }t>0\,,\\
   & Q(0) = Q_0 := \max \big\{ 1, \dx u^\varepsilon_\mrm{in} \big\}\,.
\end{align}
\end{subequations}
Solving for $ Q $ results in 
\begin{equation}
\label{unest:013}
\begin{split}
    Q(t)&=\dfrac{1}{t + Q_0^{-1}}=\frac{1}{t+\frac{1}{\max\{1,\partial_xu^\eps_{\mathrm{in}}\}}}
    <\frac{1}{t} \qquad\text{ for } t > 0\,.
\end{split}    
\end{equation}
Now, we rewrite \eqref{unest:009} as
\begin{equation}\label{unest:014}
    \begin{split}
        \dt (\dx u^\varepsilon - Q) + u^\varepsilon \dx (\dx u^\varepsilon - Q)
        &+ (\dx u^\varepsilon + Q ) (\dx u^\varepsilon - Q) \\
        &- \varepsilon^2 \dx ( \dfrac{\partial_{x} (\dx u^\varepsilon - Q)}{(\vert \dx u^\varepsilon\vert^2 + \varepsilon^2)^{3/2}}) - \varepsilon \partial_{xx} (\dx u^\varepsilon - Q) = 0\,. 
    \end{split}
\end{equation}
Then taking the $ L^2 $-inner product of \eqref{unest:014} with the regularized version $\chi^\sigma_{\lbrace \dx u^\varepsilon > Q\rbrace}$ of the characteristic function 
$\chi_{\lbrace \dx u^\varepsilon > Q\rbrace}$ and repeating the arguments for the limit procedure $\sigma\to0$ as in Items 1--2, ultimately yields
\begin{equation}
    \dfrac{\mathrm{d}}{\mathrm{d}t}\int_{\R} (\dx u^\varepsilon - Q) \chi_{\lbrace \dx u^\varepsilon > Q\rbrace} \idx 
    \leq - \int_{\R} Q (\dx u^\varepsilon - Q) \chi_{\lbrace \dx u^\varepsilon > Q\rbrace} \idx\,.
\end{equation}
Therefore, integrating the above inequality in time leads to
\begin{equation}
\label{unest:015}
    \int_{\R} (\dx u^\varepsilon(t) - Q(t))^+ \idx 
    \leq e^{\int_0^t Q(s)\,\mathrm{d}s} \int_{\R} (\dx u^\varepsilon_\mrm{in} - Q_0)^+ \idx = 0\,. 
\end{equation}
Inserting \eqref{unest:013} and taking into account that $u^\eps(t,\cdot):\R\to\R$ is continuous, we arrive at Oleinik's entropy estimate
\begin{equation}\label{unest:016}
\dx u^\varepsilon (t,x) <  
    \frac{1}{t}\qquad \text{for all } t > 0\text{ and }x\in\R\,.
\end{equation}
{\bf Proof of 4., estimate \eqref{unif-bd-partialt-ueps-LinftyL1}:}
We first calculate $ \norm{\dt u^\varepsilon(t)}{L^1(\mathbb R)}$ for any $ t > 0 $. Applying $ \dt $ to \eqref{eq:b-p-app} yields
\begin{equation}\label{eq:dt-b-p}
    \begin{gathered}
        \dt (\dt u^\varepsilon) + u^\varepsilon \dx(\dt u^\varepsilon) + \dt u^\varepsilon \dx u^\varepsilon - \varepsilon^2 \dx \Big( \dfrac{\dx(\dt u^\varepsilon)}{(\vert\dx u^\varepsilon \vert^2 + \varepsilon^2)^{3/2}} \Big) 
        - \varepsilon \partial_{xx} (\dt u^\varepsilon) = 0\,.
    \end{gathered}
\end{equation}
Taking the $ L^2 $-inner product of \eqref{eq:dt-b-p} with the regularized version $ \chi^\sigma_{\lbrace \dt u^\varepsilon > 0 \rbrace} $ of the characteristic function $ \chi_{\lbrace \dt u^\varepsilon > 0 \rbrace} $ and repeating the arguments of Items 1--2 as $\sigma\to 0,$ results in 
\begin{equation}\label{unest:017}
    \dfrac{\mathrm{d}}{\mathrm{d}t}\int_{\R} \dt u^\varepsilon \chi_{\lbrace \dt u^\varepsilon > 0 \rbrace} \idx \leq 0\,.
\end{equation}
Integrating \eqref{unest:017} gives
\begin{equation}\label{unest:018}
    \int_{\R} (\dt u^\varepsilon(t))^+ \idx \leq \int_{\R} (\dt u^\varepsilon)_\mrm{in}^+ \idx\,,
\end{equation}
where
\begin{equation}\label{unest:019}
    (\dt u^\varepsilon)_\mrm{in} := - \dx({(u^\varepsilon_\mrm{in})^2}/{2}) + \dx (\dfrac{\dx u^\varepsilon_\mrm{in}}{\sqrt{\vert \dx u^\varepsilon_\mrm{in} \vert^2 + \varepsilon^2}}) + \varepsilon \partial_{xx} u^\varepsilon_\mrm{in}\,.
\end{equation}
Repeating the argument with the characteristic function $ \chi_{\lbrace \dt u^\varepsilon <0 \rbrace} $ results in an estimate for 
$(\partial_t u^\eps)^-$ akin to \eqref{unest:018} and hence we conclude  
\begin{equation}\label{unest:020}
    \int_{\R} \vert \dt u^\varepsilon(t) \vert \idx \leq \big\|(\dt u^\varepsilon)_\mrm{in}\big\|_{L^1(\mathbb R)} \qquad \text{ for all }t > 0\,.
\end{equation}
\vspace*{1ex}
{\bf Proof of 5.\ and 6.: } Thanks to the uniform estimates \eqref{unif-bd-ueps-Linfty} and \eqref{unif-bd-ueps-BV} one also finds 
\eqref{unif-bd-drifteps-LinftyL1}, i.e.\ by means of H\"older's inequality there follows 
\begin{equation*}
\|\partial_x(u^\eps)^2/2\|_{L^\infty(0,\infty;L^1(\R))}
\leq  \|u^\eps\|_{L^\infty(0,\infty;L^\infty(\R))}
\|\partial_xu^\eps\|_{L^\infty(0,\infty;L^1(\R))}
\leq \bar u\|u^\eps_{\mathrm{in}}\|_{L^1(\R)}\,.
\end{equation*} 
Thus, together with \eqref{unif-bd-partialt-ueps-LinftyL1}, one ultimately concludes \eqref{unif-bd-visceps-LinftyL1} by comparison.
\end{proof}
%
\subsection{Limit passage $\varepsilon\to0$ in the weak formulation of \eqref{ibvp-approx}}
\label{Sec-BV-sol-plast}
%
The goal of this subsection is to pass to the limit $\varepsilon\to0$ in the weak formulation \eqref{weak-ibvp-approx}. For this, suitable compactness properties will be concluded from the uniform a priori bounds \eqref{unif-bd}. While most of the terms in \eqref{weak-ibvp-approx} will converge to the limit by means of weak convergence properties, it should be noted that the identification of the limit for the quadratic term $(u^\varepsilon)^2/2$ as the square of the limit function requires strong convergence of the sequence in $L^2_{\mathrm{loc}}((0,\infty)\times\mathbb{R})$. This will be concluded with the aid of the following Aubin-Lions-type convergence result:
\begin{thm}[{\cite[Sec.\ 9, Cor.\ 6]{Sim87CSSL}}]
\label{Help-Simon}
Consider the Banach spaces $A\subset B\subset C$ with $A\subset B$ compactly. Let $1<q\leq\infty$ and consider a family $F$  of functions such that $F$ is bounded in 
$L^q((0,T);B)\cap L^1_\mathrm{loc}((0,T);A)$ and such that $\frac{\partial F}{\partial t}$ is bounded in $L^1_\mathrm{loc}(0,T;C)$. Then, the family $F$ is relatively compact in $L^p((0,T);B)$ for all $p<q$. 
\end{thm}
Based on this we will now show the existence of a limit of the approximate solutions that also satisfies the weak formulation: 
\begin{thm}[Existence of $\mathrm{BV}$-solutions for Cauchy problem {\eqref{eq:SaB}}]
\label{Ex-BV-Sol-Limit}
Let the assumptions of Thm.\ \ref{Thm-Ex-Approx-Sol} and Prop.\ \ref{EDB-UniBds} be satisfied and let $(u^\varepsilon)_\varepsilon$ be a sequence of solutions to the approximate Cauchy problems \eqref{ibvp-approx} obtained by Thm.\ \ref{Thm-Ex-Approx-Sol}. 
Further assume that the sequence of initial data $(u^\eps_{\mathrm{in}})_\eps$ satisfies the following compatibility condition 
\begin{equation}
\|\partial_t u^\eps(0)\|_{L^1(\R)}
=\Big\|\partial_x\Big(\frac{(u^\eps_{\mathrm{in}})^2}{2}-\frac{\partial_xu^\eps_{\mathrm{in}}}{\sqrt{\partial_xu^\eps_{\mathrm{in}}+\eps^2}}-\eps\partial_xu^\eps_{\mathrm{in}}\Big)\Big\|_{L^1(\R)}
\end{equation}
and suitably approximates a limit function 
\begin{equation}
u_\mathrm{in}\in L^\infty(\R)\cap\mathrm{BV}(\R)\cap L^2(\R)\,.
\end{equation}
Then, the following statements hold true:
\begin{enumerate}
\item 
There exists a (not relabeled) subsequence $(u^\varepsilon)_\varepsilon$ and a limit pair 
$(u,R)$ of the following regularity 
\begin{subequations}
\label{reg-limit}
\begin{eqnarray}
\label{reg-u}
u&\in& L^1_\mathrm{loc}(0,\infty;\mathrm{BV}_\mathrm{loc}(\mathbb{R}))\cap L^1_\mathrm{loc}(0,\infty;\mathrm{TV}(\mathbb{R})) \cap L^\infty(0,\infty;L^2(\mathbb{R}))\quad\text{ and }\\
\label{reg-u1}
u&\in&  L^\infty(0,\infty;L^\infty(\R)\cap\mathrm{BV}(\mathbb{R}))\cap\mathrm{C}([0,\infty);L^1(\R))\cap\mathrm{C}([0,\infty);L^2(\R))\,,
\\
\label{reg-R}
R&\in& L^\infty((0,\infty)\times\mathbb{R})
\cap L^\infty(0,\infty;\mathrm{BV}(\mathbb{R}))\; 
\text{ with }|R|\leq1\text{ a.e.\ in }(0,\infty)\times\mathbb{R}\,.\qquad\qquad
\end{eqnarray}
\end{subequations}
such that the following convergence statements hold true:
\begin{subequations}
\label{conv-results}
\begin{eqnarray}
\label{conv-ueps-wstarLinftyL2}
u^\varepsilon&\overset{*}{\rightharpoonup}& u\quad\text{ in }L^\infty(0,\infty;L^2(\mathbb{R}))
\,,\\
\label{conv-epspartialxueps-L2to0}
\varepsilon\|\partial_xu^\varepsilon\|_{L^2((0,T)\times\mathbb{R})}
&\to&0
\,,\\
\label{conv-Reps-wstarLinfty}
R^\varepsilon&\overset{*}{\rightharpoonup}& R\quad\text{ in }L^\infty([0,T]\times B_r)
\;\text{ for all }T>0,\,r>0
\,,\qquad\qquad\\
\label{conv-ueps-L2}
u^\varepsilon&\to& u\quad\text{ in }L^2((0,T)\times B_r)
\;\text{ for all }T>0,r>0
\,,\\
\label{conv-ueps-L2-ptwae}
\|u^\varepsilon(t)\|_{L^2(\mathbb{R})}&\to&\|u(t)\|_{L^2(\mathbb{R})}
\quad\text{ for a.a.\ }t\in\mathbb{R}\,,\\
\label{liminf-L1partialxueps}
\liminf_{\varepsilon\to0}\int_0^T\int_{B_r}|\partial_xu^\varepsilon|\,\mathrm{d}x\,\mathrm{d}t
&\geq&\int_0^T|\mathrm{D}_xu|(B_r)\,\mathrm{d}t\;\text{ for all }T>0,r>0\,,\\
\label{conv-ueps-wstarLinfty-BV}
u^\varepsilon&\overset{*}{\rightharpoonup}& u\quad\text{ in }L^\infty(0,\infty;L^\infty(\R))\text{ and }L^\infty(0,\infty;\mathrm{BV}(\mathbb{R}))\,,\\
\label{conv-Reps-LinftyBV}
(R^\eps+\eps\partial_xu^\eps)&\overset{*}{\rightharpoonup}&R\quad\text{in }L^\infty(0,T;\mathrm{BV}(B_r))\;\text{ for all }T>0,r>0\,.
\end{eqnarray}
\end{subequations}
\item 
Assume for the sequence of initial data that 
\begin{equation}
\label{conv-ini-wL1}
u^\varepsilon_{\mathrm{in}}\rightharpoonup u_{\mathrm{in}}\quad\text{ in }L^1(\mathbb{R})\,.
\end{equation}
Then, the limit pair $(u,R)$ obtained in 1.\ is a 
distributional solution of \eqref{weak-BV-uR}, i.e., it satisfies 
\begin{equation}
\label{weak-ibvp-limit}
\begin{split}
\int_0^\infty\int_{\mathbb{R}}\Big(
u\,\partial_t\phi+\big(
\frac{u^2}{2}-R
\big)\partial_x\phi
\Big)\,\mathrm{d}x\,\mathrm{d}t
+\int_{\mathbb{R}}u_{\mathrm{in}}\phi(0)\,\mathrm{d}x=0
\end{split}
\end{equation}
for all $\phi\in \mathrm{C}^1_\mathrm{c}([0,\infty)\times\mathbb{R})$. 
\item
Assume for the sequence of initial data that there even holds 
\begin{equation}
\label{conv-ini-sL2}
u^\varepsilon_{\mathrm{in}}\to u_{\mathrm{in}}\quad\text{ in }L^2(\mathbb{R})\,.
\end{equation}
Then, the limit pair $(u,R)$ obtained in 1.\ also satisfies the following energy-dissipation estimate for all 
$t\in(0,\infty)$ 
\begin{equation}
\label{en-dissip-limit}
\begin{split}
\frac{1}{2}\|u(t)\|_{L^2(\mathbb{R})}^2
+\int_0^t|\mathrm{D}_x u(\tau)|(\mathbb{R})\,\mathrm{d}\tau
\leq\frac{1}{2}\|u_\mathrm{in}\|_{L^2(\mathbb{R})}^2\,.
\end{split}
\end{equation}
\item 
The limit $u$ satisfies Oleinik's entropy condition 
\begin{equation}
\label{u-Oleinik}
\mathrm{D}_xu(t)<\frac{1}{t}\quad\text{ in the sense of distributions for all $t>0$\,.}
\end{equation}
\item 
The limit pair $(u,R)$ satisfies the compatibility condition
\begin{equation}
\label{compat-uR}
-\mathrm{D}_xR\in\partial\Psi(u)\,,
\end{equation}
where $\partial\Psi(u)$ denotes the subdifferential of the functional $\Psi$ in $u,$ for $\Psi$  given by
\begin{equation}
\label{def-Psi-0}
\Psi: X\to[0,\infty]\,,\quad \Psi(v):=\left\{
\begin{array}{cl}
\int_0^T|\mathrm{D}_xv|(\R)\,\mathrm{d}t&\text{if }v\in Y,\\
\infty&\text{otherwise},
\end{array}\right.
\end{equation}
and with the spaces $X:=L^2((0,T)\times \R)$ and $Y:=X\cap L^1(0,T;\mathrm{BV}(\R))$ in \eqref{def-Psi-0}.
\end{enumerate}
\end{thm}
\begin{proof}
{\bf Proof of the convergence statements \eqref{conv-results}: }Convergence result \eqref{conv-ueps-wstarLinftyL2} is a direct consequence of the uniform bound \eqref{unif-bd-Linfty-L2} and convergence result \eqref{conv-epspartialxueps-L2to0} directly follows from \eqref{unif-bd-viscosity}. 
\par
In order to conclude \eqref{conv-Reps-wstarLinfty} we choose subsequences $(T_n)_n,$ $(r_n)_n$ with $T_n\nearrow\infty$ and $r_n\nearrow\infty$ monotonically increasing as $n\to\infty$. For each $n\in\mathbb{N}$
the uniform bound \eqref{unif-bd-Reps-Linfty} is valid. Thus, for $n=1,$  \eqref{unif-bd-Reps-Linfty} ensures the existence of a subequence 
$(R^{{\varepsilon}_1})_{{\varepsilon}_1}\subset(R^{\varepsilon})_{\varepsilon}$ and of a limit $R_1,$ such that 
\begin{equation*}
R^{{\varepsilon}_1}\overset{*}{\rightharpoonup}R_1\quad\text{ in }L^\infty((0,T_1)\times B_{r_1})
\quad\text{ as }\varepsilon_1\to0\,.
\end{equation*}
Similarly, for $n=2,$ \eqref{unif-bd-Reps-Linfty} ensures the existence of a further subequence $(R^{\varepsilon_2})_{\varepsilon_2}\subset(R^{{\varepsilon}_1})_{{\varepsilon}_1}$ and of a limit $R_2,$ such that 
\begin{equation*}
R^{{\varepsilon}_2}\overset{*}{\rightharpoonup}R_2\quad\text{ in }L^\infty((0,T_2)\times B_{r_2})\quad\text{ as }\varepsilon_2\to0\,.
\end{equation*}
Proceeding in this way we find for each $n\in\mathbb{N}$ a subsequence 
$(R^{{\varepsilon}_n})_{{\varepsilon}_n}\subset (R^{{\varepsilon}_m})_{{\varepsilon}_m}$ for any $m\leq n$ and a limit $R_n$ such that 
\begin{equation}
\label{conv-Reps-wstarLinfty-1}
R^{{\varepsilon}_n}\overset{*}{\rightharpoonup}R_n\quad\text{ in }L^\infty((0,T_n)\times B_{r_n})
\quad\text{ as }\varepsilon_n\to0\,.
\end{equation}
It remains to identify $R_n$ with $R_m$ on $Q_{r_m}^{T_m}:=(0,T_m)\times B_{r_m}$ of any $m\leq n\in\mathbb{N}$. The above convergence \eqref{conv-Reps-wstarLinfty-1} means 
\begin{equation}
\label{conv-Reps-wstarLinfty-2}
\int_{Q_{r_n}^{T_n}}R^{{\varepsilon}_n}\phi\,\mathrm{d}x\,\mathrm{d}t
\to\int_{Q_{r_n}^{T_n}}R_n\phi\,\mathrm{d}x\,\mathrm{d}t\quad\text{ for all }\phi\in 
L^1(Q_{r_n}^{T_n})\;\text{ as }\varepsilon_n\to0\,.
\end{equation}
A function $\tilde\phi\in L^1(Q_{r_m}^{T_m})$ can be extended to an admissible test function $\phi\in L^1(Q_{r_n}^{T_n})$ for \eqref{conv-Reps-wstarLinfty-2} as follows
\begin{equation}
\label{conv-Reps-wstarLinfty-3}
\phi(t,x):=\left\{\begin{array}{cl}
\tilde\phi(t,x)&\text{if }(t,x)\in Q_{r_m}^{T_m},\\[0.5ex]
0&\text{if }(t,x)\in Q_{r_n}^{T_n}\backslash Q_{r_m}^{T_m}.
\end{array}
\right.
\end{equation}
Using such a test function  in \eqref{conv-Reps-wstarLinfty-2} we observe, on the one hand,  for all 
$\phi\in 
L^1(Q_{r_n}^{T_n})$ as in \eqref{conv-Reps-wstarLinfty-3} that 
\begin{equation*}
\int_{Q_{r_n}^{T_n}}R^{{\varepsilon}_n}\phi\,\mathrm{d}x\,\mathrm{d}t
=\int_{Q_{r_m}^{T_m}}R^{{\varepsilon}_n}\tilde\phi\,\mathrm{d}x\,\mathrm{d}t
\to\int_{Q_{r_m}^{T_m}}R_n\tilde\phi\,\mathrm{d}x\,\mathrm{d}t
=\int_{Q_{r_n}^{T_n}}R_n\phi\,\mathrm{d}x\,\mathrm{d}t
\end{equation*}
and, on the other hand, since $(R^{{\varepsilon}_n})_{{\varepsilon}_n}\subset (R^{{\varepsilon}_m})_{{\varepsilon}_m}$ for $m\leq n,$ also that 
\begin{equation*}
\int_{Q_{r_m}^{T_m}}R^{{\varepsilon}_n}\tilde\phi\,\mathrm{d}x\,\mathrm{d}t
\to \int_{Q_{r_m}^{T_m}}R_m\tilde\phi\,\mathrm{d}x\,\mathrm{d}t
\end{equation*}
for all $\tilde\phi\in L^1(Q_{r_m}^{T_m})$. Since $\mathrm{C}_0^\infty(Q_{r_m}^{T_m})$ is dense in $L^1(Q_{r_m}^{T_m})$ we conclude by the fundamental lemma of the calculus of variations that $R_n=R_m$ a.e.\ in $Q_{r_m}^{T_m}$ for any $m\leq n\in\mathbb{N}$ and thus set $R:=R_n$.
\par
In order to deduce convergence statement \eqref{conv-ueps-L2} we will invoke 
the Aubin-Lions-type result given in Theorem \ref{Help-Simon}. For this we set 
$A:=\mathrm{BV}_{\mathrm{loc}}(\mathbb{R}),$ $B:= L^2_{\mathrm{loc}}(\mathbb{R}),$ and 
$C:=(W_0^{1,\infty}(\mathbb{R}))'$. We observe that 
\begin{equation*}
W_0^{1,\infty}(\mathbb{R})\subset W_0^{1,1}(\mathbb{R})
\subset \mathrm{BV}_{\mathrm{loc}}(\mathbb{R})=A
\subset L^2_{\mathrm{loc}}(\mathbb{R})=B
\subset (W_0^{1,\infty}(\mathbb{R}))'=C
\end{equation*}
with 
\begin{equation}
\label{conv-ueps-L2-1}
A=\mathrm{BV}_{\mathrm{loc}}(\mathbb{R})\subset B= L^2_\mathrm{loc}(\mathbb{R})\;\text{ compactly\,,}
\end{equation}
because, for space dimension $d=1$ there holds 
\begin{equation*}
\mathrm{BV}(B_r)\subset L^{\alpha}(B_r)\;\text{ for all }\alpha\leq1^*
\;\text{ with }1^*:=\frac{1d}{1-d}=\infty\,,
\end{equation*}
see, e.g.\ \cite{EvGa92}[p.\ 189, Thm.\ 1]. 
The compactness of the embedding can be concluded from Rellich's embedding theorem using that any $f\in\mathrm{BV}(B_r)$ can be approximated by a sequence $(f_n)_n$ of smooth functions 
such that $\|\partial_xf_n\|_{L^1(B_r)}\to|\mathrm{D}_xf|(B_r)$. 
\par
By \eqref{conv-ueps-L2-1} and the uniform bounds \eqref{unif-bd-Linfty-L2}, \eqref{unif-bd-BVloc}, and \eqref{unif-bd-time-der} we have that 
$(u^\varepsilon)_\varepsilon$ is uniformly bounded in $L^q((0,T);B)\cap L^1_{\mathrm{loc}}((0,T);A)$ with $q=\infty$ and that $(\partial_t u^\varepsilon)_\varepsilon$ is uniformly bounded in $L^1_{\mathrm{loc}}((0,T);C)$. Thus all the assumptions of Thm.\ \ref{Help-Simon} are satisfied and the theorem thus yields strong convergence of a subsequence in $L^p((0,T);B)$ for any $p<q=\infty$.  From this, we conclude in particular the strong $L^2$-convergence, 
i.e., \eqref{conv-ueps-L2}.   
\par
Now, by convergence result \eqref{conv-ueps-L2} and Riesz' convergence theorem, there exists a further subsequence along which pointwise convergence holds true a.e.\ in $(0,\infty),$ i.e., we conclude \eqref{conv-ueps-L2-ptwae}.
\par
As for convergence result \eqref{liminf-L1partialxueps} we argue as follows: 
From convergence result \eqref{conv-ueps-L2} we conclude in particular also strong convergence in $L^1((0,T)\times B_r)$ and hence we have 
\begin{equation*}
f_\varepsilon:=\|u^\varepsilon(\cdot)-u(\cdot)\|_{L^1(B_r)}\to0\quad\text{ in }L^1(0,T)\,.
\end{equation*}
Thus, by Riesz convergence theorem, there exists a subsequence $(f_{\varepsilon_k})_k$ that converges pointwise a.e.\ in $(0,T),$ i.e., we have 
\begin{equation*}
\|u^{\varepsilon_k}(t)-u(t)\|_{L^1(B_r)}\to0\quad\text{ for a.a.\ }t\in(0,T)\,.
\end{equation*}
By the lower semicontinuity of the variation with respect to strong $L^1$-convergence we deduce that 
\begin{equation*}
\liminf_{\varepsilon\to0}\int_{B_r}|\partial_xu^\varepsilon(t)|\,\mathrm{d}x
\geq|\mathrm{D}_xu(t)|(B_r)\;\text{ for a.a.\ }t\in(0,T)\,.
\end{equation*}
By Fatou's lemma we now conclude
\begin{equation*}
\liminf_{\varepsilon\to0}\int_0^T\int_{B_r}|\partial_xu^\varepsilon(t)|\,\mathrm{d}x\,\mathrm{d}t
\geq\int_0^T\liminf_{\varepsilon\to0}\int_{B_r}|\partial_xu^\varepsilon(t)|\,\mathrm{d}x\,\mathrm{d}t
\geq\int_0^T|\mathrm{D}_xu(t)|(B_r)\,\mathrm{d}t\,,
\end{equation*}
which is \eqref{liminf-L1partialxueps}.
\par
Furthermore, convergence statement \eqref{conv-ueps-wstarLinfty-BV} directly follows from the improved uniform bounds \eqref{unif-bd-ueps-Linfty}, 
\eqref{unif-bd-ueps-L1}, and \eqref{unif-bd-ueps-BV} obtained in Proposition \ref{UniBDsImpro}. 
\par
Finally, we verify convergence statement \eqref{conv-Reps-LinftyBV}. For this, we observe that the uniform 
bounds \eqref{unif-bd-Reps-Linfty} and \eqref{unif-bd-ueps-BV} imply for all $r>0$ and all $T>0$ that 
\begin{equation}
\|R^\eps+\eps\partial_x u^\eps\|_{L^\infty(0,T;L^1(B_r))}\leq C
\quad\text{ and }\quad
\|\eps\partial_x u^\eps\|_{L^\infty(0,T;L^1(B_r))}\to0\,.
\end{equation} 
Together with the uniform bound \eqref{unif-bd-visceps-LinftyL1} we have 
\begin{equation}
\|\eps\partial_x u^\eps\|_{L^\infty(0,T;W^{1,1}(B_r))}\leq C\,.
\end{equation}
Thus, there is a subsequence converging in the weak$*$ topology of $L^\infty(0,T;\mathrm{BV}(B_r))$. 
Using that  $\|\eps\partial_x u^\eps\|_{L^\infty(0,T;L^1(B_r))}\to0$ as $\eps\to0,$ we ultimately conclude \eqref{conv-Reps-LinftyBV}. 
\\[1ex]
{\bf Proof of the regularity results \eqref{reg-limit}: }
In a first step we directly conclude from convergence results \eqref{conv-ueps-wstarLinftyL2} and 
\eqref{liminf-L1partialxueps} that $u\in L^1_\mathrm{loc}((0,\infty);\mathrm{BV}_\mathrm{loc}(\mathbb{R}))\cap L^\infty((0,\infty);L^2(\mathbb{R}))$. Moreover, in view of the bound \eqref{unif-bd-BVloc}, by mimicking the subsequence argument of the proof for \eqref{conv-Reps-wstarLinfty} with sequences $(T_n)_n,$ $(r_n)_n$ with $T_n\nearrow\infty,$ $r_n\nearrow\infty,$  
we can identify a limit measure $\mathrm{D}_xu$ on all of $(0,\infty)\times\mathbb{R}$ and obtain that 
\begin{equation*}
\begin{split}
\int_0^\infty|\mathrm{D}_xu(t)|(\mathbb{R})\,\mathrm{d}t
&=\lim_{n\to\infty}\int_0^{T_n}|\mathrm{D}_xu(t)|(B_{r_n})\,\mathrm{d}t\\
&\leq 
\lim_{n\to\infty}\liminf_{\varepsilon\to0}\int_0^{T_n}\int_{B_{r_n}}|\partial_x u^\varepsilon|\,\mathrm{d}x\,\mathrm{d}t\,,
\end{split}
\end{equation*}
where the first equality follows from the $\sigma$-finiteness of the measure $\int_0^{[\cdot]}|\mathrm{D}_xu(t)|(\cdot)\,\mathrm{d}t$ and the second inequality is due to the lower semicontinuity of the total variation. Now, for every $n\in\mathbb{N}$ there is 
an index $\varepsilon_{*n}>0$ such that for all $\varepsilon<\varepsilon_{*n}$ we have  $T_n r_n\leq1/(2\varepsilon)$ and hence, by the bound \eqref{unif-bd-BVloc} we conclude that 
\begin{equation*}
\int_0^\infty|\mathrm{D}_xu(t)|(\mathbb{R})\,\mathrm{d}t\leq c+1\,,
\end{equation*}
which shows that indeed $u\in L^1_\mathrm{loc}((0,\infty);\mathrm{TV}(\mathbb{R}))$. 
Alltogether, with the first result we conclude \eqref{reg-u}.   
\par
In order to verify regularity result \eqref{reg-u1} we invoke convergence result \eqref{conv-ueps-wstarLinfty-BV} and 
conclude that $u\in L^\infty(0,\infty;L^\infty(\R)\cap\mathrm{BV}(\mathbb{R}))$. To deduce that also 
$u\in\mathrm{C}([0,\infty);L^1(\R))$  holds true we argue by Arz\'ela-Ascoli's theorem. For this we consider the sequence of functions 
\begin{equation*}
f^r_\eps:=\int_{B_r}|u^\eps(\cdot,x)|\,\mathrm{d}x:\;[0,\infty)\to[0,\infty)
\end{equation*} 
with $B_r=(-r,r)$ and for all $r>0$ we observe that 
\begin{enumerate}
\item 
for all $\eps>0$ it is $f^r_\eps\in\mathrm{C}([0,\infty))$ by regularity property \eqref{reg-ueps}; 
\item 
the sequence $(f^r_\eps)_\eps$ is uniformly bounded pointwise in $[0,\infty)$ by the uniform bound \eqref{unif-bd-ueps-Linfty};
\item 
the sequence $(f^r_\eps)_\eps$ is uniformly equicontinuous, which can be seen as follows:
\end{enumerate}
Fix $\kappa>0$. We have to show that there is a $\delta>0$ such that for all $\eps>0$ and for all $s,t$ such that $|s-t|<\delta$ there holds 
$|f^r_\eps(s)-f^r_\eps(t)|<\kappa$. Indeed, this holds true, because
\begin{equation}
\label{reg-u1C}
\begin{split}
|f^r_\eps(s)-f^r_\eps(t)|
\leq\int_{\R}|u^\eps(s,x)-u^\eps(t,x)|\,\mathrm{d}x
&\leq \int_{\R}|\partial_tu^\eps(\xi_{st},x)||s-t|\,\mathrm{d}x\\
&\leq \|\partial_tu^\eps\|_{L^\infty(0,\infty;L^1(\R))}|s-t|\leq C|s-t|\,.
\end{split}
\end{equation}
where the last estimate is due to the uniform bound \eqref{unif-bd-partialt-ueps-LinftyL1}. 
Moreover, $\xi_{st}\in[s,t]$ in the second estimate 
denotes a suitable intermediate value obtained by the mean value theorem.  
Restricting the sequence $(f^r_\eps)_\eps$ to a compact interval $[0,T_n]$ for $T_n>0,$ we conclude by  Arz\'ela-Ascoli's theorem that a subsequence converges uniformly on $[0,T_n]$ to a limit function $f^r_n\in\mathrm{C}([0,T_n])$. 
On the other hand, we have by convergence property \eqref{conv-ueps-L2} that $f^r_\eps\to\|u(\cdot)\|_{L^1(B_r)}$  in $L^1(0,T_n)$. By the uniqueness of the limit we therefore conclude that $u\in\mathrm{C}([0,T_n];L^1(B_r))$. Since estimate \eqref{reg-u1C} is independent of 
$r>0$ and $T_n>0$ we can consider a sequence of radii $r\to\infty$ and a sequence $T_n\to\infty$ as $n\to\infty$ 
to conclude that $u\in\mathrm{C}([0,\infty);L^1(\R))$. 
\par 
By similar arguments one can also show that $u\in\mathrm{C}([0,\infty);L^2(\R))$. For this we introduce the sequence of functions
\begin{equation}
g_\eps:=\int_{B_r}|u^\eps(\cdot,x)|^2\,\mathrm{d}x:\;[0,\infty)\to[0,\infty)
\end{equation}
and again check the preconditions of Arz\'ela-Ascoli's theorem:
\begin{enumerate}
\item for all $\eps>0$ it is $g^r_\eps\in\mathrm{C}([0,\infty))$ by regularity property \eqref{reg-ueps}; 
\item 
the sequence $(g_\eps)_\eps$ is uniformly bounded pointwise in $[0,\infty)$ by the uniform bound  
\eqref{unif-bd-Linfty-L2};
\item 
the sequence $(g_\eps)_\eps$ is uniformly equicontinuous, which can be seen with the aid of \eqref{reg-u1C} as follows:
\end{enumerate}
\begin{equation}
\begin{split}
\big|\|u^\eps(s)\|_{L^2(\R)}-\|u^\eps(t)\|_{L^2(\R)}\big|
&\leq \|u^\eps(s)-u^\eps(t)\|_{L^2(\R)}\\
&\leq\|u^\eps(s)-u^\eps(t)\|_{L^\infty(\R)}\,\|u^\eps(s)-u^\eps(t)\|_{L^1(\R)}\\
&\leq2\|u^\eps\|_{L^\infty(0,\infty;L^\infty(\R))}\int_{\R}|u^\eps(s,x)-u^\eps(t,x)|\,\mathrm{d}x\\
&\leq C|s-t|\,,
\end{split}
\end{equation}
where we used \eqref{reg-u1C} and the uniform bound \eqref{unif-bd-ueps-Linfty} to conclude the last estimate. Repeating the arguments from above with the sequence $(g_\eps)_\eps$ restricted to a sequence of compact intervals $[0,T_n]$ one finds uniform convergence with Arz\'ela-Ascoli's theorem and can ultimately conclude that $u\in\mathrm{C}([0,\infty);L^2(\R))$. 
\par
Next, we verify the regularity results for $R$ stated in \eqref{reg-R}. The first regularity result in \eqref{reg-R} is concluded from convergence result  \eqref{conv-Reps-wstarLinfty}, also using that 
$|R^\varepsilon|\leq 1$ in $(0,\infty)\times\mathbb{R}$ for all $\varepsilon>0$. Indeed, assume that there exists a set 
$M_1\subset(0,\infty)\times\mathbb{R}$ of positive Lebesgue measure $\mathcal{L}^2(M_0)>0,$ such that $R>1$ on $M_1$. Then, convergence result \eqref{conv-ueps-wstarLinftyL2} leads to a contradiction when testing with the indicator function $\chi_{M_1}$ of the set $M_1,$ i.e., 
\begin{equation*}
0>\int_0^\infty\int_\mathbb{R}(R^\varepsilon-1)\chi_{M_1}\,\mathrm{d}x\,\mathrm{d}t
\to\int_0^\infty\int_\mathbb{R}(R-1)\chi_{M_1}\,\mathrm{d}x\,\mathrm{d}t
=\int_{M_1}(R-1)\,\mathrm{d}x\,\mathrm{d}t>0\,.
\end{equation*} 
In a similar manner it can be verified that the limit also satisfies $R\geq-1$. 
\par
Finally, the second regularity result in \eqref{reg-R}, i.e., that $R\in L^\infty(0,\infty;\mathrm{BV}(\R)),$ is obtained by invoking convergence result \eqref{conv-Reps-LinftyBV} on a sequence of growing domains $(0,T_n)\times B_{r_n}\to(0,\infty)\times\R$ as $n\to\infty$ by repeating the previous arguments. 
\\[1ex]
{\bf Proof of 2.\ The weak formulation \eqref{weak-ibvp-limit} for the limit pair $(u,R)$: }Next, we verify that the limit pair is a 
weak solution in the sense of \eqref{weak-ibvp-limit}, given that \eqref{conv-ini-wL1} holds true for the initial data. For this we aim to pass to the limit in the weak formulation   \eqref{weak-ibvp-approx}. 
Indeed, as $\varepsilon\to0$ we arrive at \eqref{weak-ibvp-limit} by making use of the convergence results \eqref{conv-ueps-wstarLinftyL2}, \eqref{conv-epspartialxueps-L2to0}, \eqref{conv-Reps-wstarLinfty}, and \eqref{conv-ueps-L2}, and also thanks to assumption  \eqref{conv-ini-wL1}. 
\\[1ex]
{\bf Proof of 3.\ Energy-dissipation estimate \eqref{en-dissip-limit}: } We obtain \eqref{en-dissip-limit} by passing to the limit $\varepsilon\to0$ in the energy-dissipation balance \eqref{en-dissip-eps}, on the left-hand side   
by means of the weak lower semicontinuity of the total variation-term using convergence result \eqref{liminf-L1partialxueps} and by the strong $L^2$-convergence \eqref{conv-ueps-L2-ptwae} pointwise a.e. in time, and on the right-hand side by exploiting assumption \eqref{conv-ini-sL2}. Using that $u\in\mathrm{C}([0,\infty);L^2(\R))$ allows us to infer that \eqref{en-dissip-limit} holds true even for all $t\in[0,T]$. 
Indeed, assume that there is $t_*>0$ such that \eqref{en-dissip-limit} is violated. Observe that the terms on the left-hand side of \eqref{en-dissip-limit} are continuous with respect to $t\in[0,\infty)$ thanks to $u\in\mathrm{C}([0,\infty);L^2(\R))$ and the absolute continuity of the integral term. Therefore, by continuity, there is a whole neighbourhood $B_\delta(t_*)$ such that \eqref{en-dissip-limit} is violated 
for all $t\in B_\delta(t_*)$.  However, we have already deduced that \eqref{en-dissip-limit} holds true for a.a.\ $t>0,$ thus, in 
particular for a.a.\  $t\in B_\delta(t_*),$ which states the contradiction. This concludes the proof of \eqref{en-dissip-limit}. 
\\[1ex]
{\bf Proof of 4.\ Oleinik's entropy condition \eqref{u-Oleinik}: } 
The proof of \eqref{u-Oleinik} is given in Section \ref{Sec:Limit-Oleinik}.
\\[1ex]
{\bf Proof of 5.\ Compatibility condition \eqref{compat-uR}: }
For 
this proof we refer to Section \ref{Sec:CharR}. 
\end{proof}
%
\subsection{Limit passage in Oleinik's entropy condition \eqref{unif-bd-ueps-Oleinik}} 
\label{Sec:Limit-Oleinik}
%
Due to the regularity of the limit $u\in L^1(0,T;\mathrm{BV}(\R))\cap L^\infty(0,T;L^2(\R))$ and by the 
convergence result \eqref{conv-ueps-L2},  
we are able to pass to the limit in Oleinik's entropy condition \eqref{unif-bd-ueps-Oleinik} in distributional sense, only, cf.\ \eqref{limit-Oleinik1}. 
However, in order to characterize the jump behavior of the limit solution along jump curves in a better way, 
we also derive from \eqref{unif-bd-ueps-Oleinik} a relation for the difference quotients, 
which can then be passed to the limit thanks to convergence property \eqref{conv-ueps-L2} and thus results in statement \eqref{limit-Oleinik2}. 
\begin{proposition}[Oleinik's entropy condition for a solution $u$ of the limit system]
\label{Limit-Oleinik}
Let the assumptions of Theorem \ref{Ex-BV-Sol-Limit} be satisfied. The limit solution $u$ obtained in Theorem \ref{Ex-BV-Sol-Limit} satisfies Oleinik's entropy condition in the sense of distributions, i.e., 
\begin{subequations}
\begin{equation}
\label{limit-Oleinik1}
-\int_{\R}u(t)\partial_x\phi\,\mathrm{d}x<\int_{\R}\frac{\phi}{t}\,\mathrm{d}x\qquad\text{for all }\phi\in \mathrm{C}_0^\infty(\R)\;\text{ and all }t>0\,.
\end{equation}
Moreover, there also holds
\begin{equation}
\label{limit-Oleinik2}
u(t,x+h)-u(t,x-h)<\frac{2h}{t}+4Ch\quad\text{for a.a.\ }x\in\R\,,\;\text{for a.a.\ }h>0\,,\;\text{ and for all }t>0\,,
\end{equation}
\end{subequations}
where $C>0$ stems from the uniform bound \eqref{unif-bd-ueps-BV}. 
\end{proposition}
\begin{proof}
{\bf Proof of \eqref{limit-Oleinik1}: }
Let $t>0$ be arbitrary but fixed. Testing \eqref{unif-bd-ueps-Oleinik} by $\phi\in \mathrm{C}_0^\infty(\R)$ and integrating by parts in space on the left-hand side of the inequality 
results in 
\begin{equation*}
\int_{\R}\phi\,\partial_xu^\eps(t)\,\mathrm{d}x=-\int_{\R}u^\eps(t)\partial_x\phi\,\mathrm{d}x<\int_{\R}\frac{\phi}{t}\,\mathrm{d}x\,.
\end{equation*}
Thanks to convergence result \eqref{conv-ueps-L2} we can pass to the limit $\eps\to0$ in the above equation and first obtain \eqref{limit-Oleinik1} for a.e.\ $t>0$. Using that the limit $u$ satisfies $u\in\mathrm{C}([0,\infty);L^1(\R))$ and that also the term on the right-hand side is continuous in $t>0$ allows us conclude that \eqref{limit-Oleinik1} holds true for all $t>0$.
\\[1ex]
{\bf Proof of \eqref{limit-Oleinik2}: }
For all $\eps>0$ and all $t>0$ we have $u^\eps(t,\cdot)\in \mathrm{C}^1(\R)$ by regularity property \eqref{reg-ueps}. Let now 
$h>0$ 
and consider an arbitrary 
\begin{equation}
\label{limit-Oleinik21}
\phi\in \mathrm{C}_0^\infty(\R)\;\text{ 
with $\|\phi\|_\infty=1$}\,.
\end{equation}
Then we deduce the following estimate from Oleinik's entropy condition  \eqref{unif-bd-ueps-Oleinik}
\begin{equation}
\label{limit-Oleinik22}
\begin{split}
&\int_{\R}\frac{u^\eps(t,x+h)-u^\eps(t,x-h)}{h}\phi\,\mathrm{d}x\\
&<\int_{\R}\Big(\frac{2}{t}+\Big|2\partial_xu^\eps(t,x)-\frac{u^\eps(t,x+h)-u^\eps(t,x-h)}{h}\Big|\Big)\phi\,\mathrm{d}x\\
&\leq\int_{\R}\Big(\frac{2}{t}+2|\partial_xu^\eps(t,x)|+|\partial_xu^\eps(t,x+\tilde h_+)|+|\partial_xu^\eps(t,x-\tilde h_-)|\Big)\phi\,\mathrm{d}x\\
&\leq\int_{\R}\frac{2\phi}{t}\,\mathrm{d}x+4\|\partial_xu^\eps(t,\cdot)\|_{L^1(\R)}\|\phi\|_{\infty}
\leq\int_{\R}\frac{2\phi}{t}\,\mathrm{d}x+4C\,.
\end{split}
\end{equation}
Here we used the mean value theorem of differentiability with suitable intermediate values $\tilde h_\pm\in (0,h)$ in the second estimate, and the uniform bound \eqref{unif-bd-ueps-BV} to arrive at the last estimate.   Making use of the strong $L^2$-convergence 
\eqref{conv-ueps-L2} for a.a.\ $t>0,$ 
we can pass to the limit $\eps\to0$ in \eqref{limit-Oleinik22} and thus obtain 
\begin{equation}
\label{limit-Oleinik23}
\begin{split}
\int_{\R}\frac{u(t,x+h)-u(t,x-h)}{h}\phi\,\mathrm{d}x
<\int_{\R}\frac{2\phi}{t}\,\mathrm{d}x+4C\quad\text{ for all }\phi\in\mathrm{C}_0^\infty(\R) \text{ with }\eqref{limit-Oleinik21}\,,
\end{split}
\end{equation}
at first, for a.e.\ $t\in(0,\infty)$. Arguing again that $u\in\mathrm{C}([0,\infty);L^1(\R))$ like above for \eqref{limit-Oleinik1}, we find that \eqref{limit-Oleinik23} holds true for all $t>0$.  
By the regularity property \eqref{reg-u1} for $u$ we ultimately conclude \eqref{limit-Oleinik2} from \eqref{limit-Oleinik23}. 
\end{proof}
%
\section{Characterization of the limit $R$}
\label{Sec:CharR}
%
In the following we want to characterize the limit $R$ obtained in Thm.\ \ref{Ex-BV-Sol-Limit}, 1.\  by relating it to an element of the subdifferential of the total variation functional $\Psi$ defined on a suitable Bochner space. More precisely, for the limit pair $(u,R)$ our goal is to show that the negative of the distributional derivative of $R$ is an element of the total variation functional evaluated in $u,$ i.e.
\begin{equation}
\label{char-R0}
\zeta\in\partial\Psi(u)\quad\text{ and }\quad\zeta=-\mathrm{D}_xR\;\text{ in the sense of distributions\,.}
\end{equation}
 To obtain this relation we first extend in Section \ref{Sec:Weak-V} the weak formulation in the sense of $\mathrm{BV}$-solutions 
\eqref{weak-ibvp-limit} to a weak formulation in the dual of the space  
$V:=\big(L^\infty(0,T;L^\infty(\R)\cap\mathrm{BV}(\mathbb{R}))\big)$. This will allow us to use the solution $u$ as a test function in this weak formulation. Based on this, we subsequently deduce the characterization \eqref{char-R0} of $R$ in Section \ref{Sec:CharR-Detail}.
%
\subsection{Weak formulation in $V^*:=\big(L^\infty(0,T;L^\infty(\R)\cap\mathrm{BV}(\mathbb{R}))\big)^*$}
\label{Sec:Weak-V}
%
Recall that estimates \eqref{unif-bd-ueps-Linfty}, \eqref{unif-bd-ueps-L1}, \eqref{unif-bd-ueps-BV}, \eqref{unif-bd-partialt-ueps-LinftyL1}, 
\eqref{unif-bd-drifteps-LinftyL1}, and \eqref{unif-bd-visceps-LinftyL1} provide the uniform bounds 
\begin{subequations}
\label{uni-bd-12}
\begin{align}
\label{uni-bd-1}
\|u^\eps\|_{L^\infty(L^\infty)}+\|u^\eps\|_{L^\infty(L^1)}+\|\partial_xu^\eps\|_{L^\infty(L^1)}&\leq C\,,\\
\label{uni-bd-2}
\|\partial_tu^\eps\|_{L^\infty(L^1)}+\|\partial_x(u^\eps)^2/2\|_{L^\infty(L^1)}
+\|\partial_x R^\eps+\eps\partial_x^2u^\eps\|_{L^\infty(L^1)}&\leq C\,.
\end{align}
\end{subequations}
By \eqref{uni-bd-1} the sequence $(u^\eps)_\eps$ is uniformly bounded in 
\begin{equation*}
V:=L^\infty(0,T;L^\infty(\R)\cap\mathrm{BV}(\mathbb{R}))\subset L^\infty(0,T;L^\infty(\R))\,.
\end{equation*}  
We want to argue now that the plastic Burgers equation \eqref{burgers-plastic} even holds true in 
\begin{equation*}
V^*=\big(L^\infty(0,T;L^\infty(\R)\cap\mathrm{BV}(\mathbb{R}))\big)^*\,.
\end{equation*}
For this, let $\xi_\eps\in\{\partial_tu^\eps,\partial_x(u^\eps)^2/2,\partial_xR^\eps+\eps\partial_xu^\eps\}$ 
and observe that, by Hahn-Banach's extension theorem and \eqref{uni-bd-2} it is 
\begin{equation*}
\underset{\|v\|_V=1}{\sup_{v\in V}}\langle\xi_\eps,v\rangle_V
=\underset{\|v\|_V=1}{\sup_{v\in V}}\langle\xi_\eps,v\rangle_{L^\infty(0,T;L^\infty(\R))}
\leq\underset{\|v\|_V=1}{\sup_{v\in V}}\|\xi_\eps\|_{L^\infty(L^1)}\|v\|_{L^\infty(L^\infty)}\leq C\,,
\end{equation*}
which gives the following uniform bound in $V^*$
\begin{equation}
\label{uni-bd+}
\|\partial_tu^\eps\|_{V^*}+\|\partial_x(u^\eps)^2/2\|_{V^*}
+\|\partial_x R^\eps+\eps\partial_x^2u^\eps\|_{V^*}\leq C\,.
\end{equation}
Thus, by Banach-Alaoglu's theorem there exist elements $\alpha,\beta,\gamma\in V^*$ such that
\begin{subequations}
\label{conv-abc}
\begin{align}
\label{conv-abc-a}
\partial_tu_\eps&\overset{*}{\rightharpoonup}\alpha\quad\text{ in }V^*\,,\\
\label{conv-abc-b}
\partial_x(u_\eps)^2/2&\overset{*}{\rightharpoonup}\beta\quad\text{ in }V^*\,,\\
\label{conv-abc-c}
\partial_x R_\eps+\eps\partial_x^2u^\eps&\overset{*}{\rightharpoonup}\gamma\quad\text{ in }V^*\,
\end{align}
\end{subequations}
along a (not-relabeled) subsequence. 
Therefore we can state the following theorem, which also provides a better identification of the above three limits:  
\begin{thm}
\label{Weak-V}
Let $(u^\eps)_\eps\subset V$ be the sequence of approximating solutions obtained in Theorem \eqref{Thm-Ex-Approx-Sol}, $(u,R)$ be the limit pair identified in Theorem 
\ref{Ex-BV-Sol-Limit}, and $\alpha,\beta,\gamma\in V^*$ such that also convergence statements \eqref{conv-abc} hold true. 
Then the weak formulation of the Cauchy problem \eqref{CauchyP-burgers-plastic} of the plastic Burgers equation holds true in $V^*$ as follows
\begin{equation}
\label{weak-V}
\langle\alpha,v\rangle_V+\langle\beta,v\rangle_V+\langle\gamma,v\rangle_V=0\quad\text{ for all }v\in V\,,
\end{equation}
where 
\begin{align}
\label{weak-V-id}
\alpha=\mathrm{D}_tu\,,\;\;
\beta=\mathrm{D}_xu^2/2=u\mathrm{D}_xu\,,\;\;\text{and }
\gamma=-\mathrm{D}_xR
\quad\text{ in the sense of distributions}\,,
\end{align}
so that \eqref{weak-V} is given by \eqref{weak-ibvp-limit} 
for all $\phi\in \mathrm{C}^1_{\mathrm{c}}((0,T)\times\R)$. In addition to regularity properties \eqref{reg-limit}, the function $u$ also satisfies
\begin{equation}
\label{reg-u+}
u\in W^{1,\infty}(0,T;(L^\infty(\R)\cap\mathrm{BV}(\mathbb{R}))^*)\,.
\end{equation}
Moreover, for $v=u$ in \eqref{weak-V} the following identification holds true:
\begin{subequations}
\label{weak-V-id-u}
\begin{align}
\label{weak-V-id-ua}
\langle\alpha,u\rangle_V&=\frac{1}{2}\|u(T)\|_{L^2(\R)}^2-\frac{1}{2}\|u_{\mathrm{in}}\|_{L^2(\R)}^2\,,\\
\label{weak-V-id-ub}
\langle\beta,u\rangle_V&=0\,.
\end{align}
\end{subequations}
\end{thm}
\begin{proof}
{\bf Proof of \eqref{weak-V}: }
The uniform bound \eqref{uni-bd+} and convergence statements \eqref{conv-abc} allow it to test 
equation \eqref{CauchyP-burgers-plastic} by elements $v\in V$. 
Passing to the limit $\eps\to0$ in this weak formulation in $V^*$ gives   \eqref{weak-V-id}. 
\\[1ex]
{\bf Proof of the identification \eqref{weak-V-id}: }The first of \eqref{weak-V-id} follows from 
\begin{subequations}
\label{id-a}
\begin{align}
\langle \partial_tu^\eps,\phi\rangle_V
=\int_0^T\int_{\R}\partial_tu^\eps\phi\,\mathrm{d}x\,\mathrm{d}t&\to\langle \alpha,\phi\rangle_V\quad\text{ and }\\
\int_0^T\int_{\R}u^\eps\partial_t\phi\,\mathrm{d}x\,\mathrm{d}t&\to\int_0^T\int_{\R}u\partial_t\phi\,\mathrm{d}x\,\mathrm{d}t
\end{align}
\end{subequations}
for all $\phi\in \mathrm{C}^\infty_{\mathrm{c}}((0,T)\times\R)$ thanks to convergence results \eqref{conv-abc-a} and \eqref{conv-ueps-L2}, which shows that $\langle\alpha,\phi\rangle_V=-\int_0^T\int_\R u\partial_t\phi\,\mathrm{d}x\,\mathrm{d}t$. 
\par 
For the second result in \eqref{weak-V-id} we repeat above arguments to find 
$\langle \beta,\phi\rangle_V=-\int_0^T\int_\R \frac{u^2}{2}\partial_x\phi\,\mathrm{d}x\,\mathrm{d}t$ for all $\phi\in \mathrm{C}^\infty_{\mathrm{c}}((0,T)\times\R)$. Additionally, the uniform bound $\|\partial_x(u^\eps)^2/2\|_{V^*}\leq C$ also allows us to conclude the existence of an element $\hat\beta\in V^*$ such that
\begin{equation*}
u^\eps\partial_x u^\eps\overset{*}{\rightharpoonup}\hat\beta\quad\text{in }V^*\,.
\end{equation*}
Repeating once more the argument of \eqref{id-a} for $(u^\eps\partial_x u^\eps)_\eps$ and $\hat\beta,$ also using the convergence result \eqref{conv-ueps-wstarLinfty-BV}, shows that
\begin{equation*}
\langle \beta,\phi\rangle_V=\lim_{\eps\to0}\langle (\partial_x u^\eps)^2/2,\phi\rangle_V
=\lim_{\eps\to0}\langle u^\eps\partial_x u^\eps,\phi\rangle_V=\langle \hat\beta,\phi\rangle_V
\end{equation*}
for all $\phi\in \mathrm{C}^\infty_{\mathrm{c}}((0,T)\times\R)$, which gives the second of \eqref{weak-V-id}.
\par
Now, for the third statement in \eqref{weak-V-id} we also argue like above and use in addition that 
\begin{equation*}
\big|\langle-\eps\partial_x^2u^\eps,\phi\rangle_V\big|=\Big|\int_0^T\int_{\R}\eps\partial_xu^\eps\partial_x\phi\,\mathrm{d}\,\mathrm{d}t\Big|
\to0\quad\text{ as $\eps\to0$}
\end{equation*}
thanks to convergence result \eqref{conv-epspartialxueps-L2to0}. 
\\[1ex]
{\bf Proof of regularity property \eqref{reg-u+}: }
It follows from the bounds \eqref{uni-bd-12} and from the identification  \eqref{weak-V-id} that 
$\mathrm{D}_tu\in L^\infty(0,T;(L^\infty(\R)\cap\mathrm{BV}(\mathbb{R}))^*)$. Regularity \eqref{reg-u} provides $u\in L^\infty(0,T;L^\infty(\R)\cap\mathrm{BV}(\mathbb{R}))=V\subset V^*$ and we conclude \eqref{reg-u+}. 
\\[1ex]
{\bf Proof of the identification \eqref{weak-V-id-u}: }
To find \eqref{weak-V-id-ua} we note that $W=(L^\infty(\R)\cap\mathrm{BV}(\R))\subset L^2(\R)\subset (L^\infty(\R)\cap\mathrm{BV}(\R))^*=W^*$ forms a Gelfand triple 
with $(L^\infty(\R)\cap\mathrm{BV}(\R))\subset L^2(\R)$ continuously and densely. 
Then $W^{1,p,p'}(0,T;W,W^*):=\{v\in L^p(0,T;W),\,
\partial_t v\in L^{p'}(0,T;W^*)\}\subset \mathrm{C}(0,T;L^2(\R))$ and for all $v,w\in W^{1,p,p'}(0,T;W,W^*),$ $t_1\leq t_2\in [0,T]$ the following integration-by-parts formula holds true, cf.\ \cite[L.\ 7.3]{RoubicekNLPDE}
\begin{equation*}
\int_{\R}v(t_2)w(t_2)\,\mathrm{d}x-\int_{\R}v(t_1)w(t_1)\,\mathrm{d}x
=\int_{t_1}^{t_2}\big(\langle\mathrm{D}_tv,w\rangle_W+\langle\mathrm{D}_tw,v\rangle_W\big)\,\mathrm{d}t\,.
\end{equation*}
For $v=w=u$ this gives   \eqref{weak-V-id-ua}. 
\par
In order to verify \eqref{weak-V-id-ub} consider $\hat u_r:=u\chi_{[0,T]\times B_r}$ with 
$\chi_{[0,T]\times B_r}\in L^\infty(0,T;V)$   
the characteristic function of the set $[0,T]\times B_r$. For every $r>0$ we have
\begin{equation}
\label{weak-V-id-ub1}
\begin{split}
\langle\beta,\hat u_r\rangle_V=\langle\mathrm{D}_x(u^2/2),\hat u_r\rangle_V
&=\langle\mathrm{D}_x(\hat u_r^3/3),\chi_{[0,T]\times B_r}\rangle_V\\
&=\langle-\hat u_r^3/3,0\rangle_V+
\int_0^T\Big(\frac{\hat u_r^3}{3}(t,r)-\frac{\hat u_r^3}{3}(t,-r)\Big)\,\mathrm{d}t\,.
\end{split}
\end{equation}
Accordingly, \eqref{weak-V-id-ub} holds true, if  
\begin{equation}
\label{weak-V-id-ub2}
\int_0^T\Big(\frac{\hat u_r^3}{3}(t,r)-\frac{\hat u_r^3}{3}(t,-r)\Big)\,\mathrm{d}t\to0\quad\text{as }r\to\infty
\end{equation}
in \eqref{weak-V-id-ub1}. To show \eqref{weak-V-id-ub2} we proceed by contradiction. 
Since 
\begin{equation*}
\langle\beta,u\rangle_V=\lim_{r\to\infty}\langle\beta,\hat u_r\rangle_V
=\lim_{r\to\infty}\int_0^T\Big(\frac{\hat u_r^3}{3}(t,r)-\frac{\hat u_r^3}{3}(t,-r)\Big)\,\mathrm{d}t
\end{equation*}
is a well-defined limit, it suffices to show that the assumption
\begin{equation}
\label{weak-V-id-ub3}
\exists\,c>0\;\;\forall r>0:\;\;\int_0^T\Big|\frac{\hat u_r^3}{3}(t,r)\Big|+\Big|\frac{\hat u_r^3}{3}(t,-r)\Big|\,\mathrm{d}t>c
\end{equation}
leads to a contradiction. Indeed, by regularity result \eqref{reg-u1} we have for a.a.\ $r>0$
\begin{equation}
\label{weak-V-id-ub4}
\begin{split}
\int_0^T\Big|\frac{\hat u_r^3}{3}(t,r)\Big|+\Big|\frac{\hat u_r^3}{3}(t,-r)\Big|\,\mathrm{d}t
&\leq\int_0^\infty\int_0^T\Big|\frac{\hat u_r^3}{3}(t,r)\Big|+\Big|\frac{\hat u_r^3}{3}(t,-r)\Big|\,\mathrm{d}t\,\mathrm{d}r\\
&\leq\int_0^T\int_{\R}\Big|\frac{u^3}{3}(t,x)\Big|\,\mathrm{d}x\,\mathrm{d}t\\
&\leq \frac{T}{3}\|u\|_{L^\infty(0,T;L^\infty(\R))}^2\|u\|_{L^\infty(0,T;L^1(\R))}
\end{split}
\end{equation}
and therefore, putting \eqref{weak-V-id-ub3} and \eqref{weak-V-id-ub4} together, we arrive at the following contradiction:
\begin{align*}
\infty>\int_0^T\int_{\R}\Big|\frac{\hat u_r^3}{3}(t,x)\Big|\,\mathrm{d}x\,\mathrm{d}t
\geq&\int_0^\infty\int_0^T\Big|\frac{\hat u_r^3}{3}(t,r)\Big|+\Big|\frac{\hat u_r^3}{3}(t,-r)\Big|\,\mathrm{d}t\,\mathrm{d}r\\
\geq& 
\int_0^\infty c\,\mathrm{d}r=\infty\,.
\end{align*}
Thus, assumption \eqref{weak-V-id-ub3} is false and \eqref{weak-V-id-ub2} has to be true, which concludes the proof of \eqref{weak-V-id-ub}.
\end{proof}
%
\subsection{Characterization of the limit $R$ in $Y^*:=X\cap L^1(0,T;\mathrm{BV}(\R))$}
\label{Sec:CharR-Detail}
%
In the following we want to characterize the limit $R$ obtained in Thm.\ \ref{Ex-BV-Sol-Limit} by relating it to an element of the subdifferential of the total variation functional. In view of the regularity result \eqref{reg-u}, we introduce 
\begin{equation}
\label{def-Psi}
\begin{split}
\Psi&:X\to[0,\infty]\,,\\
\Psi(u)&:=
\left\{\begin{array}{cl}
\int_0^T|\mathrm{D}_xu|(\mathbb{R})\,\mathrm{d}t
&\text{ if }u\in Y,\\
\infty&\text{ otherwise,}
\end{array}\right.\\
\text{where }X&:=L^{2}((0,T);L^2(\mathbb{R}))
\;\text{ and }\;Y:=X\cap L^1(0,T;\mathrm{BV}(\R))\,.
\end{split}
\end{equation}
We note that this functional 
\begin{equation}
\label{prop-Psi}
\Psi \text{ is positively $1$-homogeneous and satisfies $\Psi(0)=0$. }
\end{equation}
\par 
The identification of $R$ will now be achieved with the aid of the following result from convex analysis, cf.\ e.g.\ \cite[Lemma 1.3.1]{MieRou15RIST}: 
\begin{proposition}
\label{Char-Subdiff-1Homog}
Let $\Phi:X\to[0,\infty]$ be a positively $1$-homogeneous functional on the space $X$ with $\Phi(0)=0$. Then, for any $v\in X$ there holds
\begin{equation}
\label{char-subdiff-1homog}
\eta\in\partial\Phi(v)
\qquad\Leftrightarrow\qquad
\left\{
\begin{array}{lrl}
\text{i)}&\eta&\!\!\in\partial\Phi(0),\\[0.5ex]
\text{ii)}&\langle\eta,v\rangle_X&\!\!=\Phi(v).
\end{array}
\right.
\end{equation}
\end{proposition} 
In Theorem \ref{Char-Ri} we will show that the limit pair $(u,R)$ satisfies property 
\eqref{char-subdiff-1homog}, i) for the functional $\Psi$ 
from \eqref{def-Psi} and in Theorem \ref{Char-Rii}  that also property 
\eqref{char-subdiff-1homog}, ii) holds true. More precisely, with these two theorems we obtain the existence of an element 
$\zeta\in X^*$ with the properties
\begin{equation*}
\zeta\in\partial\Psi(u)\quad\text{ and }\quad\zeta=-\mathrm{D}_xR\;\text{ in the sense of distributions.}
\end{equation*}
Recall that $|R|\leq1$ a.e.\ in $(0,T)\times\R$ by \eqref{reg-R}. Therefore the obtained characterization of $R$ is in line with results in literature that give a characterization of the subdifferential of the total variation functional, 
cf.\ e.g.\ \cite[Prop.\ 3.1]{ChGoNo15FPSC}, \cite[Thm.\ 26]{AndMaz05TVF}, \cite[p.\ 482f]{BeCaNo02TVFR}.  
\begin{thm}[Characterization of the limit $R$, part I: \eqref{char-subdiff-1homog}, i)]
\label{Char-Ri}
Set
\begin{equation*}
X:=L^2(0,T;L^2(\R))
\quad\text{and}\quad
Y:=X\cap L^1(0,T;\mathrm{BV}(\R))\,.
\end{equation*}
Let $(u,R)$ be the limit pair obtained from the sequence of approximate solutions $(u^\varepsilon)_\eps$ by Thm.\ \ref{Ex-BV-Sol-Limit} and consider the functional $\Psi:X\to[0,\infty]$ from \eqref{def-Psi}. Then, the following statements hold true:
\begin{subequations}
\begin{eqnarray}
\label{elem-subdiff-eps}
-\partial_x R^\varepsilon\in\partial\Psi(0)\quad\text{for all }\varepsilon>0\,,
\end{eqnarray}
and there is a subsequence and a limit 
\begin{equation}
\label{conv-Reps2a}
\zeta\in X^*
\end{equation} 
such that
\begin{eqnarray}
\label{conv-Reps2b}
-\partial_x R^\varepsilon&\overset{*}{\rightharpoonup}&\zeta\quad\text{ in }Y^*\,,\\
\label{id-limit1}
\zeta&=&-\mathrm{D}_x R\quad\text{in the sense of distributions}\,,\\
\label{id-limit2}
\zeta&\in&\partial\Psi(0)\,.
\end{eqnarray}
\end{subequations}
\end{thm}
\begin{proof}
{\bf Proof of \eqref{elem-subdiff-eps}: }
Let $v\in X=L^2(0,T;L^2(\R))$. Since $\mathrm{C}^\infty(0,T;\R)$ is dense in $X$ we find for every $\kappa>0$ and for every $\eps>0$ an element 
$v_\kappa^\eps\in \mathrm{C}^\infty(0,T;\R)$ such that 
\begin{subequations}
\label{elem-subdiff-eps0001}
\begin{equation}
\label{elem-subdiff-eps00011}
\big| \langle-\partial_xR^\varepsilon,v\rangle_{X}-\langle-\partial_xR^\varepsilon,v_\kappa^\eps\rangle_{X}\big|
\leq \|-\partial_xR^\eps\|_{X}\|v-v_\kappa^\eps\|_X\leq C(\eps,T)\|v-v_\kappa^\eps\|_X\leq\kappa\,,
\end{equation}
by estimate \eqref{uni-bd-sol-DCP-smooth9} as well as 
\begin{equation}
\label{elem-subdiff-eps00012}
\big| \langle-\eps\partial_x^2u^\eps,v\rangle_{X}-\langle-\eps\partial_x^2u^\eps,v_\kappa^\eps\rangle_{X}\big|
\leq \eps\|-\partial_x^2u^\eps\|_{X}\|v-v_\kappa^\eps\|_X\leq C(\eps,T)\|v-v_\kappa^\eps\|_X\leq\kappa\,,
\end{equation}
\end{subequations}
by estimate \eqref{uni-bd-sol-DCP-smooth5}. 
Thanks to \eqref{elem-subdiff-eps00011}  we have for all $\kappa>0$ and all $\eps>0$
\begin{subequations}
\begin{equation}
\label{elem-subdiff-eps001}
\begin{split}
\langle-\partial_xR^\varepsilon,v\rangle_{X}
&\leq \langle-\partial_xR^\varepsilon,v_\kappa^\eps\rangle_{X}+\kappa
=\int_0^T\int_{\R}
R_\eps\partial_x v^\eps_\kappa \,\mathrm{d}t+\kappa\\
&\leq\int_0^T\int_{\R}|R^\eps|\,|\partial_xv_\kappa^\eps|\,\mathrm{d}x\,\mathrm{d}t+\kappa
\leq\Psi(v_\kappa^\eps)+\kappa\,,
\end{split}
\end{equation}
since $|R^\eps|=\frac{|\partial_xu^\varepsilon|}{\sqrt{|\partial_xu^\varepsilon|^2+\varepsilon^2}}\leq1$ in $(0,T)\times\mathbb{R}$. Similarly we also find  with \eqref{elem-subdiff-eps00012} for all $\kappa>0$ and all $\eps>0$
\begin{equation}
\label{elem-subdiff-eps0012}
\begin{split}
\langle-\eps\partial_x^2u^\varepsilon,v\rangle_{X}
&\leq\langle-\eps\partial_x^2u^\varepsilon,v_\kappa^\eps\rangle_{X}+\kappa
=\int_0^T\int_{\R}\eps\partial_xu^\varepsilon\partial_xv_\kappa^\eps\,\mathrm{d}x\,\mathrm{d}t+\kappa\\
&\leq \sqrt{\eps}\|\partial_xu^\varepsilon\|_{L^2((0,T)\times\R)}\,\sqrt{\eps}\|\partial_xv^\varepsilon_\kappa\|_{L^2((0,T)\times\R)}+\kappa\\
&\leq C\,\sqrt{\eps}\|\partial_xv^\varepsilon_\kappa\|_{L^2((0,T)\times\R)}+\kappa
\end{split}
\end{equation}
\end{subequations}
thanks to estimate \eqref{unif-bd-viscosity}. 
Moreover, since $\mathrm{C}^\infty(0,T;\R)$ is also dense in $L^2(0,T;\mathrm{BV}(\R))$ with respect to strict convergence, for each $\eps>0$ and $\kappa>0,$ and for every $v\in L^2(0,T;\mathrm{BV}(\R))$ a function $v^\eps_\kappa\in \mathrm{C}^\infty(0,T;\R)$ can be found such that additionally also
\begin{equation}
\label{elem-subdiff-eps01}
\big|\Psi(v^\eps_\kappa)-\Psi(v)\big|\leq\kappa\,\quad\text{for all $v\in L^2(0,T;\mathrm{BV}(\R))$. }
\end{equation}
Since \eqref{elem-subdiff-eps001} and \eqref{elem-subdiff-eps01} hold true for all $\eps>0$ and for all $\kappa>0$ we also conclude that 
\begin{equation}
\label{elem-subdiff-eps1}
\langle-\partial_xR^\varepsilon,v\rangle_{X}\leq\Psi(v)\quad\text{ for all }v\in X\,.
\end{equation}
This follows from the fact that  $\Psi(v)=\infty$ for any $v\in X\backslash L^2(0,T;\mathrm{BV}(\R))$. This proves \eqref{elem-subdiff-eps}. 
\par
However, note that estimates \eqref{elem-subdiff-eps00012} and \eqref{elem-subdiff-eps0012} do not allow us to conclude a 
uniform estimate alike \eqref{elem-subdiff-eps1}. This is due to the fact that, according to \eqref{uni-bd-sol-DCP-smooth44}, the constant $C(\eps,T)$ in \eqref{elem-subdiff-eps00012} blows up exponentially as $\eps\to0$. This must be compensated by $\|v-v_\kappa^\eps\|_X$ small enough for given $\kappa>0$ in \eqref{elem-subdiff-eps00012}. Since $v^\eps_\kappa$ is obtained from $v$ by convolution with 
a mollifier, this can be achieved by choosing the support of the mollifier suitably small. In turn, the gradient $\partial_xv^\eps_\kappa$ will be very large, so that its $L^2$-norm may be rather related to the exponential blow-up in $\eps$ and cannot be controlled by   $\sqrt{\eps}$ in \eqref{elem-subdiff-eps0012}.  
\\[1ex]
{\bf Proof of \eqref{conv-Reps2a} and  \eqref{conv-Reps2b}: } Set $Y:=X\cap L^1(0,T;\mathrm{BV}(\R))\subset X$. Using \eqref{elem-subdiff-eps1} we find
\begin{equation}
\underset{\|v\|_Y\leq1}{\sup_{v\in Y}}\langle-\partial_xR^\varepsilon,v\rangle_{X}\leq\Psi(v)\leq\|v\|_Y\leq1\,.
\end{equation}
Thus, by Banach-Alaoglu's theorem we conclude the existence of a (not relabelled) subsequence
\begin{equation*}
-\partial_xR^\varepsilon\overset{*}{\rightharpoonup} \tilde\zeta\quad\text{ in }Y^*\,,
\end{equation*}
which is  \eqref{conv-Reps2b}. 
The space $Y$ is a linear subspace of the normed vector space $X$.  Thus, by Hahn-Banach's extension theorem, the 
continuous linear functional $\tilde\zeta:Y\to\R$ can be extended to a continuous linear functional $\zeta:X\to\R$ 
such that $\|\tilde\zeta\|_{Y^*}=\|\zeta\|_{X^*}$. Therefore it holds
\begin{equation}
\label{conv-Reps21}
\langle\zeta,v\rangle_X=\langle\zeta,v\rangle_Y\quad\text{ for all }v\in Y\,.
\end{equation}
This proves \eqref{conv-Reps2a}. 
\\[1ex]
{\bf Proof of \eqref{id-limit1}: }For every $\varepsilon>0$ and for all $\phi\in \mathrm{C}_0^\infty(\mathbb{R})$ there holds
\begin{equation}
\int_0^T\int_{\R}\partial_xR^\varepsilon\phi\,\mathrm{d}x\,\mathrm{d}t
=-\int_0^T\int_{\R} R^\varepsilon\partial_x\phi\,\mathrm{d}x\,\mathrm{d}t\,,
\end{equation}
where
\begin{subequations}
\begin{eqnarray}
\label{id-limit11}
\int_0^T\int_{\R}\partial_xR^\varepsilon\phi\,\mathrm{d}x\,\mathrm{d}t&\to&\langle\zeta,\phi\rangle_Y=\langle\zeta,\phi\rangle_X\,,\\
-\int_0^T\int_{\R} R^\varepsilon\partial_x\phi\,\mathrm{d}x\,\mathrm{d}t&\to&
-\int_0^T\int_{\R} R\partial_x\phi\,\mathrm{d}x\,\mathrm{d}t
\end{eqnarray}
\end{subequations}
by weak-strong convergence arguments. Note that the equality in \eqref{id-limit11} follows by \eqref{conv-Reps21}. 
This proves \eqref{id-limit1}.
\\[1ex]
{\bf Proof of \eqref{id-limit2}: }We have to show that 
\begin{equation}
\label{def-dPsi0}
\langle\zeta,v\rangle_{L^2}\leq\Psi(v)\quad\text{for all }v\in X\,.
\end{equation}
For this, observe that for all $\varepsilon>0$ and all $v\in Y$ there holds
\begin{equation}
\Psi(v)\geq\langle-\partial_xR^\varepsilon,v\rangle_X=\langle-\partial_xR^\varepsilon,v\rangle_Y\to\langle\zeta,v\rangle_Y=\langle\zeta,v\rangle_X\,.
\end{equation}
Moreover, if $v\in X\backslash Y,$ then $\Psi(v)=\infty$ and hence inequality \eqref{def-dPsi0} is trivially satisfied. 
\end{proof}
With the following theorem we also verify property \eqref{char-subdiff-1homog}, ii): 
\begin{thm}[Characterization of the limit $R$, part II: \eqref{char-subdiff-1homog}, ii)]
\label{Char-Rii}
Let $(u,R)$ be the limit pair obtained from the sequence of approximate solutions $(u^\varepsilon)_\eps$ by Thm.\ \ref{Ex-BV-Sol-Limit} and consider the functional $\Psi$ from \eqref{def-Psi}. Let $\gamma\in V^*$ be the limit of the subsequence $(-\partial_xR^\eps-\eps\partial_x^2u^\eps)_\eps$ obtained in Theorem \ref{Weak-V} and let $\zeta\in X^*$ be the limit of the subsequence 
$(-\partial_xR^\eps)_\eps$ obtained in Theorem \ref{Char-Ri}. Then the following statements hold true: 
\begin{subequations}
\begin{align}
\label{id-limit3}
\langle\zeta,u\rangle_X&=\Psi(u)\,,\\
\label{id-limit4}
\gamma&=\zeta=-\mathrm{D}_xR\quad\text{ in the sense of distributions\,,}\\
\label{id-limit5}
\gamma&=\zeta\quad\hspace*{9.5ex}\text{ in }V^*\,.
\end{align}
\end{subequations}
\end{thm}
\begin{proof}
{\bf Proof of \eqref{id-limit4}: }Observe that \eqref{id-limit1} provides $\zeta=-\mathrm{D}_xR$ in the sense of distributions and \eqref{weak-V-id} 
provides $\gamma=-\mathrm{D}_xR$ in the sense of distributions, so that \eqref{id-limit4} follows immediately.
\\[1ex]
{\bf Proof of \eqref{id-limit5}: }It is $V:=L^\infty(0,T;L^\infty(\R)\cap\mathrm{BV}(\R))\subset X:=L^2((0,T)\times \R)$ densely. By convergence results \eqref{conv-abc}, \eqref{conv-Reps2b}, and by the uniqueness of the limit in $V^*$ we therefore conclude 
\begin{equation*}
\langle\gamma,v\rangle_V=\langle\zeta,v\rangle_X\quad\text{for all }v\in V\,,
\end{equation*}
which is \eqref{id-limit5}\,.
\\[1ex]
{\bf Proof of \eqref{id-limit3}: } In order to show relation \eqref{id-limit3} we verify the following chain of inequalities
\begin{equation}
\label{id-limit30}
\begin{split}
\Psi(u)
\overset{(1)}{\leq}\liminf_{\varepsilon\to0}\langle-\partial_xR^\varepsilon,u^\varepsilon\rangle_{X}
&\overset{(2)}{\leq}\limsup_{\varepsilon\to0}\langle-\partial_xR^\varepsilon,u^\varepsilon\rangle_{X}\\
&\overset{(3)}{=}\limsup_{\varepsilon\to0}\langle-\partial_xR^\varepsilon,u^\varepsilon\rangle_{V}\\
&\overset{(4)}{\leq}\langle\gamma,u\rangle_{V}\\
&\overset{(5)}{=}\langle\zeta,u\rangle_{X}+\big(\langle\gamma,u\rangle_{V}-\langle\zeta,u\rangle_{X}\big)\\
&\overset{(6)}{\leq}\Psi(u)\,.
\end{split}
\end{equation}
We now give the detailed arguments for each of the steps (1)-(6) in estimate \eqref{id-limit30}. 
\par
Estimate (1) follows by the lower semicontinuity of the total variation after carrying out an integration by parts and the following estimate 
\begin{equation}
\label{id-limit31}
\begin{split}
\langle-\partial_xR^\varepsilon,u^\varepsilon\rangle_{X}&=\int_0^T\int_{\R}-\partial_xR^\varepsilon\,u^\varepsilon\,\mathrm{d}x\,\mathrm{d}t\\
&\overset{\text{(1a)}}{=}\int_0^T\int_{\R}R^\varepsilon\,\partial_xu^\varepsilon\,\mathrm{d}x\,\mathrm{d}t
=\int_0^T\int_{\R}\frac{|\partial_xu^\varepsilon|^2}{\sqrt{|\partial_xu^\varepsilon|^2+\eps^2}}\,\mathrm{d}x\,\mathrm{d}t\\
&\geq\int_0^T\int_{B_r}\frac{|\partial_xu^\varepsilon|^2}{\sqrt{|\partial_xu^\varepsilon|^2+\eps^2}}\,\mathrm{d}x\,\mathrm{d}t
\overset{\text{(1b)}}{\geq}\int_0^T\int_{B_r}|\partial_xu^\varepsilon|\,\mathrm{d}x\,\mathrm{d}t-2r\eps T\,
\end{split}
\end{equation}
for all $r>0$. 
Above, the integration by parts in (1a) makes use of the far-field boundary condition \eqref{int-far-field3} and (1b) 
is deduced with the same argument as the uniform bound \eqref{unif-bd-BVloc}.  
Thanks to the lower semicontinuity result \eqref{liminf-L1partialxueps} estimate \eqref{id-limit31} allows us to conclude
\begin{equation}
\label{id-limit311}
\begin{split}
\liminf_{\eps\to0}\langle-\partial_xR^\varepsilon,u^\varepsilon\rangle_{X}
\geq\liminf_{\eps\to0}\Big(\int_0^T\int_{B_r}|\partial_xu^\varepsilon|\,\mathrm{d}x\,\mathrm{d}t-2r\eps T\Big)
\geq\int_0^T|\mathrm{D}_xu|(B_r)\,\mathrm{d}t
\end{split}
\end{equation}
for all $r>0$. Letting $r\to\infty$ in \eqref{id-limit311} and using the continuity from below of the measure, proves (1).
\par 
Now, estimate (2) in \eqref{id-limit30} is a direct consequence of the properties of the limit inferior and the limit superior. For relation (3) we observe that for every $\eps>0$ it is 
\begin{equation*}
\langle-\partial_xR^\varepsilon,u^\varepsilon\rangle_{X}=\langle-\partial_xR^\varepsilon,u^\varepsilon\rangle_{V}
\end{equation*}
by the regularity properties \eqref{reg-ueps}. Hence (3) follows. 
\par 
Next we verify estimate (4). For this, we make use of the equation \eqref{eq:b-p-app} to rewrite the term on the left-hand side  of (4) as follows
\begin{equation}
\label{id-limit34}
\begin{split}
\langle-\partial_xR^\varepsilon,u^\varepsilon\rangle_{V}
&=-\langle \partial_tu^\eps+\partial_x\frac{(u^\eps)^2}{2}-\eps\partial_x^2u^\eps,u^\eps\rangle_V\\
&=-\int_0^T\int_{\R}\Big(\partial_tu^\eps+\partial_x\frac{(u^\eps)^2}{2}\Big)u^\eps\,\mathrm{d}x\,\mathrm{d}t
+\int_0^T\int_{\R}\eps\big(\partial_x^2u^\eps\big) u^\eps\,\mathrm{d}x\,\mathrm{d}t\\
&=-\int_0^T\int_{\R}\Big(\partial_t\frac{(u^\eps)^2}{2}+\partial_x\frac{(u^\eps)^3}{3}\Big)\,\mathrm{d}x\,\mathrm{d}t
-\int_0^T\int_{\R}\eps(\partial_xu^\eps)^2\,\mathrm{d}x\,\mathrm{d}t\\
&\leq\frac{1}{2}\|u^\eps_{\mathrm{in}}\|_{L^2(\R)}^2-\frac{1}{2}\|u^\eps(T)\|_{L^2(\R)}^2-0-0\,.
\end{split}
\end{equation}
Here we used an integration by parts in space on the term $\int_0^T\int_{\R}\eps\big(\partial_x^2u^\eps\big) u^\eps\,\mathrm{d}x\,\mathrm{d}t$ to get from the second into the third line, also making use of the far-field boundary condition \eqref{int-far-field2}. The resulting term has a negative sign and is therefore estimated from above by $0$ when passing to the fourth line. Similarly, we also used an integration by parts in space for the term $-\int_0^T\int_{\R}\partial_x\frac{(u^\eps)^3}{3}\,\mathrm{d}x\,\mathrm{d}t$ together with the far-field boundary condition \eqref{int-far-field3} and an integration by parts in time for the term $-\int_0^T\int_{\R}\partial_t\frac{(u^\eps)^2}{2}\,\mathrm{d}x\,\mathrm{d}t$ to arrive at the fourth line. 
Thanks to the well-preparedness of the initial data \eqref{conv-ini-sL2} and convergence result \eqref{conv-ueps-L2-ptwae}, taking the limit superior on both sides of \eqref{id-limit34}, gives 
\begin{equation}
\label{id-limit341}
\limsup_{\eps\to0}\,\langle-\partial_xR^\varepsilon,u^\varepsilon\rangle_{V}
\leq\frac{1}{2}\|u_{\mathrm{in}}\|_{L^2(\R)}^2-\frac{1}{2}\|u(T)\|_{L^2(\R)}^2-0\,.
\end{equation}
At this point we make use of Theorem \ref{Weak-V}, in particular of the limit equation \eqref{weak-V} and the identification relations \eqref{weak-V-id-u} to rewrite the right-hand side of \eqref{id-limit341} as follows
\begin{equation}
\label{id-limit342}
\begin{split}
\limsup_{\eps\to0}\,\langle-\partial_xR^\varepsilon,u^\varepsilon\rangle_{V}
&\leq
\frac{1}{2}\|u_{\mathrm{in}}\|_{L^2(\R)}^2-\frac{1}{2}\|u(T)\|_{L^2(\R)}^2-0\\
&=-\langle\alpha,u\rangle_V-\langle\beta,u\rangle_V
=\langle\gamma,u\rangle_V\,,
\end{split}
\end{equation}
which is \eqref{id-limit34}. 
\par
Now (5) is obtained by adding and subtracting the term $\langle\zeta,u\rangle_{X}$. Finally (6) follows by applying \eqref{id-limit2}, 
i.e., the information 
$\zeta\in\partial\Psi(0)$ for the specific test function $v=u$ and by noting that the term 
\begin{equation*}
\big(\langle\gamma,u\rangle_{V}-\langle\zeta,u\rangle_{X}\big)=0
\end{equation*}
thanks to $u\in V$ and the identification \eqref{id-limit5}. This concludes the proof of estimate \eqref{id-limit3}.  
\end{proof}
%
\section{Existence of approximate solutions, Proof of Theorem \ref{Thm-Ex-Approx-Sol}}
\label{Sec-Pf-Ex-Approx-Sol}
%
In this section we give the proof of Theorem \ref{Thm-Ex-Approx-Sol}. The proof will be carried out in four steps
\begin{itemize}
\item[1.] Local-in-time solutions for a Dirichlet problem on a finite interval $B_r:=(-r,r)$ for any $r>0$ on any time interval $[0,T]$ (Proposition \ref{Ex-sol-DCP-smooth});
\item[2.] Global a priori estimates uniformly in $r$ (Proposition \ref{Prop-uni-bd-r});
\item[3.] Limit passage $r\to\infty$ (Proposition \ref{Prop-limit-r});
\item[4.] Verification of the integral far-field relations for the limit solutions (Proposition \ref{Prop-Far-field}).
\end{itemize}
{\bf To Step 1: }Here, for any $\eps>0$ fixed, the following Cauchy problem with homogeneous Dirichlet conditions is considered on $Q_{r}^T:=(0,T)\times B_r$ for arbitrary $T,r>0$: Find $u^{\eps,r}:[0,T]\times B_r\to\R$ such that
\begin{subequations}
\label{DCP-smooth}
\begin{alignat}{2}
\label{DCP-smooth1}
\partial_tu^{\eps,r}+\partial_x\Big(\frac{(u^{\eps,r})^2}{2}-R^{\eps,r}-\eps\partial_xu^{\eps,r}\Big)&=0&&\quad\text{ in }Q_r^T\,,\\
u^{\eps,r}(t,\pm r)&=0&&\quad\text{ for }t\in(0,T)\,,\\
u^{\eps,r}(0,x)&=u^{\eps,r}_{\mathrm{in}}(x)&&\quad\text{ for }x\in B_r\,,\\
\text{and with }R^{\eps,r}:=\frac{\partial_xu^{\eps,r}}{\sqrt{|\partial_xu^{\eps,r}|^2+\eps^2}}\;\;\text{ in }\eqref{DCP-smooth1}\,.&&
\end{alignat}
\end{subequations}
For Problem \eqref{DCP-smooth} one can show the following result: 
\begin{proposition}[Existence of solutions of \eqref{DCP-smooth}]
\label{Ex-sol-DCP-smooth}
Keep $\eps>0$ fixed. For all $r>0$ assume that $u^{\eps,r}_{\mathrm{in}}\in H^2(B_r)$ with $u^{\eps,r}_{\mathrm{in}}(\pm r)=0$. 
Then, for every $T>0$ and every $r>0$ there exists a solution $u^{\eps,r}:[0,T]\times B_r\to\R$ of problem  \eqref{DCP-smooth} of the regularity 
\begin{equation}
\label{reg-sol-r}
\begin{split}
u^{\eps,r}&\in L^\infty(0,T;H^2(B_r))\cap L^2(0,T;H^2(B_r))\quad\text{ with }\\
\partial_tu^{\eps,r}&\in L^\infty(0,T;L^2(B_r))\cap L^2(0,T;H^1(B_r))\,.
\end{split} 
\end{equation}
\end{proposition}
\begin{proof}
Problem \eqref{DCP-smooth} can be solved by a time-discretization and a Galerkin method making use of the  a priori estimates akin to the ones deduced below in Prop.\ \ref{Prop-uni-bd-r}. We omit the details of the proof here. 
\end{proof}
{\bf To Step 2: Global a priori estimates uniformly in $r$: }
\begin{proposition}[Global a priori estimates of the solutions $(u^{\eps,r})_r$ uniformly in $r$]
\label{Prop-uni-bd-r}
Let the assumptions of Prop.\ \ref{Ex-sol-DCP-smooth} hold true. Let $\eps>0$ and $T>0$ be fixed. Further assume that there is a constant $C_0>0$ such that  the sequence of initial data 
$(u^{\eps,r}_{\mathrm{in}})_r$ satisfies the following bound uniformly in $r>0$:
\begin{subequations}
\begin{equation}
\label{uni-bd-init-r1}
\|u^{\eps,r}_{\mathrm{in}}\|_{L^2(B_r)}+\|\partial_xu^{\eps,r}_{\mathrm{in}}\|_{L^2(B_r)}\leq C_0\,,
\end{equation}
and assume that there also holds 
\begin{equation}
\label{uni-bd-init-r2}
\|\partial_tu^{\eps,r}(0)\|_{L^2(B_r)}\leq C_0\,.
\end{equation}
\end{subequations}
Then there are constants $C>0$ and $C(\eps,T)>0$ such that the solutions $(u^{\eps,r})_r$ of problem 
\eqref{DCP-smooth} satisfy the following estimates  uniformly for all $r>0$: 
\begin{subequations}
\label{uni-bd-sol-DCP-smooth}
\begin{align}
\label{uni-bd-sol-DCP-smooth1}
\|u^{\eps,r}(t)\|_{L^2(B_r)}^2&\leq C \quad\text{for all }t\in[0,T]\,,
\\
\label{uni-bd-sol-DCP-smooth2}
\eps\|\partial_x u^{\eps,r}\|_{L^2((0,T)\times B_r)}^2&\leq C\,,
\\
\label{uni-bd-sol-DCP-smooth3}
\eps\int_0^T\|u^{\eps,r}(t)\|_{L^\infty(B_r)}^2\,\mathrm{d}t&\leq C\,,
\\
\label{uni-bd-sol-DCP-smooth4}
\eps\|\partial_xu^{\eps,r}(t)\|_{L^2(B_r)}^2&\leq C(\eps,T)\quad\text{for all }t\in[0,T]\,,
\\
\label{uni-bd-sol-DCP-smooth5}
\int_0^T\|\partial_x^2u^{\eps,r}(t)\|_{L^2(B_r)}^2\,\mathrm{d}t&\leq C(\eps,T)\,,
\\
\label{uni-bd-sol-DCP-smooth6}
\int_0^T\|\partial_xu^{\eps,r}(t)\|_{L^\infty(B_r)}^2\,\mathrm{d}t&\leq C(\eps,T)\,,
\\
\label{uni-bd-sol-DCP-smooth7}
\|\partial_tu^{\eps,r}(t)\|_{L^2(B_r)}^2&\leq C(\eps,T)\quad\text{for all }t\in[0,T]\,,
\\
\label{uni-bd-sol-DCP-smooth8}
\|\partial_x^2u^{\eps,r}(t)\|_{L^2(B_r)}^2&\leq C(\eps,T)\quad\text{for all }t\in[0,T]\,,
\\
\label{uni-bd-sol-DCP-smooth9}
\int_0^T\|\partial_x^2R^{\eps,r}\|_{L^2(B_r)}^2\,\mathrm{d}t&\leq C(\eps,T)
\quad\text{for }\;R^{\eps,r}:=\frac{\partial_xu^{\eps,r}}{\sqrt{|\partial_xu^{\eps,r}|^2+\eps^2}}
\end{align}
\end{subequations}
\end{proposition}
\begin{proof}
{\bf Proof of \eqref{uni-bd-sol-DCP-smooth1} -- \eqref{uni-bd-sol-DCP-smooth3}: }
Multiplying \eqref{DCP-smooth1} with $ u^{\eps,r},$ integrating over $B_r$ and performing an integration by parts results in 
\begin{equation}
\label{uni-bd-sol-DCP-smooth11}
\frac{\mathrm{d}}{\mathrm{d}t}\Big(\frac{1}{2}\|u^{\eps,r}(t)\|_{L^2(B_r)}^2\Big)
+\eps\|\partial_x u^{\eps,r}(t)\|_{L^2(B_r)}^2
+\int_{B_r}\frac{|\partial_xu^{\eps,r}|^2}{\sqrt{|\partial_xu^{\eps,r}|^2+\eps^2}}\,\mathrm{d}x=0\,.
\end{equation} 
Integrating \eqref{uni-bd-sol-DCP-smooth11} in time yields for all $T'\leq T$ that
\begin{equation}
\label{uni-bd-sol-DCP-smooth12}
\|u^{\eps,r}(T')\|_{L^2(B_r)}^2+\eps\int_0^{T'}\|\partial_x u^{\eps,r}(t)\|_{L^2(B_r)}^2\,\mathrm{d}t
\leq \|u^{\eps,r}(0)\|_{L^2(B_r)}^2\,,
\end{equation}
which provides the bounds \eqref{uni-bd-sol-DCP-smooth1} and \eqref{uni-bd-sol-DCP-smooth2} thanks to the uniform bound on the initial data \eqref{uni-bd-init-r1}. 
From \eqref{uni-bd-sol-DCP-smooth12} together with the Sobolev imbedding inequality one further obtains
\begin{equation}
\eps\int_0^T\|u^{\eps,r}(t)\|_{L^\infty(B_r)}^2\,\mathrm{d}t\leq C_{\mathrm{Sob}}\eps\int_0^T\|\partial_x u^{\eps,r}(t)\|_{L^2(B_r)}^2\,\mathrm{d}t
\leq C\,,
\end{equation}
which is \eqref{uni-bd-sol-DCP-smooth3}.
\\[1ex]
{\bf Proof of \eqref{uni-bd-sol-DCP-smooth4} -- \eqref{uni-bd-sol-DCP-smooth6}: } 
Now, we rewrite \eqref{DCP-smooth1} as follows:
\begin{equation}
\label{uni-bd-sol-DCP-smooth41}
\partial_tu^{\eps,r}-\partial_x\big(R^{\eps,r}+\partial_xu^{\eps,r}\big)=-\partial_x\big((u^{\eps,r})^2/2\big)\,.
\end{equation}
Taking the square of \eqref{uni-bd-sol-DCP-smooth41}, integrating the resultant over $B_r,$ and applying an integration by parts, gives
\begin{equation}
\label{uni-bd-sol-DCP-smooth42}
\begin{split}
&\frac{\mathrm{d}}{\mathrm{d}t}\Big(
\eps\|\partial_xu^{\eps,r}(t)\|_{L^2(B_r)}^2+2\!\int_{B_r}\!\!\sqrt{|\partial_xu^{\eps,r}(t)|^2+\eps^2}\,\mathrm{d}x
\Big)\\
&\hspace*{17ex}
+\|\partial_tu^{\eps,r}(t)\|_{L^2(B_r)}^2
+\frac{\eps^2}{2}\|\partial_x^2u^{\eps,r}(t)\|_{L^2(B_r)}^2
+\frac{1}{2}\|\partial_xR^{\eps,r}(t)\|_{L^2(B_r)}^2\\
&=\|u^{\eps,r}(t)\partial_xu^{\eps,r}(t)\|_{L^2(B_r)}^2
\leq \|u^{\eps,r}(t)\|_{L^\infty(B_r)}^2 \|\partial_xu^{\eps,r}(t)\|_{L^2(B_r)}^2\\
&\leq \|u^{\eps,r}(t)\|_{L^\infty(B_r)}^2\frac{\eps}{\eps}\Big( \|\partial_xu^{\eps,r}(t)\|_{L^2(B_r)}^2
+\frac{2}{\eps}\int_{B_r}\!\!\sqrt{|\partial_xu^{\eps,r}(t)|^2+\eps^2}\,\mathrm{d}x\Big)\,.
\end{split}
\end{equation}
Applying Gr\"onwall's inequality in \eqref{uni-bd-sol-DCP-smooth42} yields for all $T'\leq T$
\begin{equation}
\label{uni-bd-sol-DCP-smooth43}
\begin{split}
&\Big(
\eps\|\partial_xu^{\eps,r}(T')\|_{L^2(B_r)}^2+2\!\int_{B_r}\!\!\sqrt{|\partial_xu^{\eps,r}(T')|^2+\eps^2}\,\mathrm{d}x
\Big)\\
&+\int_0^{T'}\Big(
\exp\big(-\int_0^{t}\frac{\|u^{\eps,r}(s)\|_{L^\infty(B_r)}^2}{\eps}\,\mathrm{d}s\big)\\
&\hspace*{15ex}
\Big(\|\partial_tu^{\eps,r}(t)\|_{L^2(B_r)}^2+\frac{\eps^2}{2}\|\partial_x^2u^{\eps,r}(t)\|_{L^2(B_r)}^2
+\frac{1}{2}\|\partial_xR^{\eps,r}(t)\|_{L^2(B_r)}^2\Big)\Big)
\,\mathrm{d}t\\
&\leq\Big(
\eps\|\partial_xu^{\eps,r}(0)\|_{L^2(B_r)}^2+2\!\int_{B_r}\!\!\sqrt{|\partial_xu^{\eps,r}(0)|^2+\eps^2}\,\mathrm{d}x
\Big)\exp\big(\int_0^{T}\frac{\|u^{\eps,r}(t)\|_{L^\infty(B_r)}^2}{\eps}\,\mathrm{d}t\big)\\
&\leq C(\eps,T)
\end{split}
\end{equation}
thanks to  \eqref{uni-bd-sol-DCP-smooth3} and the uniform bound on the initial data \eqref{uni-bd-init-r1}.  
This proves \eqref{uni-bd-sol-DCP-smooth4}. 
Estimate \eqref{uni-bd-sol-DCP-smooth43} in particular also yields that
\begin{equation}
\label{uni-bd-sol-DCP-smooth44}
\begin{split}
&\int_0^T\frac{\eps^2}{2}\|\partial_x^2u^{\eps,r}(t)\|_{L^2(B_r)}^2\,\mathrm{d}t\\
&\leq\Big(
2\eps\|\partial_xu^{\eps,r}(0)\|_{L^2(B_r)}^2+2\!\int_{B_r}\!\!\sqrt{|\partial_xu^{\eps,r}(0)|^2+\eps^2}\,\mathrm{d}x
\Big)\exp\big(2\int_0^{T}\frac{\|u^{\eps,r}(t)\|_{L^\infty(B_r)}^2}{\eps}\,\mathrm{d}t\big)\\
&\leq C(\eps,T)\,,
\end{split}
\end{equation}
which proves \eqref{uni-bd-sol-DCP-smooth5}. Estimate \eqref{uni-bd-sol-DCP-smooth6} 
then follows from \eqref{uni-bd-sol-DCP-smooth5}, resp.\ \eqref{uni-bd-sol-DCP-smooth44}, by applying the Sobolev imbedding inequality. 
\\[1ex]
{\bf Proof of \eqref{uni-bd-sol-DCP-smooth7}: }Next, we apply $\partial_t$ to both sides of \eqref{DCP-smooth1}. 
Denoting $u^{\eps,r}_t:=\partial_tu^{\eps,r}$ the resulting equation can be rewritten as 
\begin{equation}
\label{uni-bd-sol-DCP-smooth71}
\partial_tu_t^{\eps,r}+u^{\eps,r}\partial_xu_t^{\eps,r}+u_t^{\eps,r}\partial_xu^{\eps,r}-\partial_x\Big(\partial_tR^{\eps,r}-\eps\partial_xu_t^{\eps,r}\Big)=0\,.
\end{equation}
Testing \eqref{uni-bd-sol-DCP-smooth71} with $u^{\eps,r}_t$, integrating over $B_r$, and applying an integration by parts, leads to
\begin{equation}
\label{uni-bd-sol-DCP-smooth72}
\begin{split}
&\frac{\mathrm{d}}{\mathrm{d}t}
\left\{
\frac{1}{2}\|u_t^{\eps,r}\|_{L^2(B_r)}^2
\right\}
+\eps\|\partial_xu_t^{\eps,r}\|_{L^2(B_r)}^2
+\int_{B_r}\frac{\eps^2|\partial_xu_t^{\eps,r}|^2}{(|\partial_xu_t^{\eps,r}|^2+\eps^2)^{3/2}}\,\mathrm{d}x\\
&=-\frac{1}{2}\int_{B_r}\partial_xu^{\eps,r}|u_t^{\eps,r}|^2\,\mathrm{d}x
\leq\frac{1}{2}\|\partial_xu^{\eps,r}\|_{L^\infty(B_r)}\,\|u_t^{\eps,r}\|_{L^2(B_r)}^2\,.
\end{split}
\end{equation} 
Applying Gr\"onwall's inequality to \eqref{uni-bd-sol-DCP-smooth72} yields for all $T'\leq T$ that
\begin{equation}
\begin{split}
&\left\{
\frac{1}{2}\|u_t^{\eps,r}(T')\|_{L^2(B_r)}^2
\right\}\\
&\qquad+\int_0^{T'}
\exp\big(-\int_0^s
\|\partial_xu^{\eps,r}(s)\|_{L^\infty(B_r)}
\,\mathrm{d}s\big)\\
&\hspace*{18ex}
\Big(\eps\|\partial_xu_t^{\eps,r}\|_{L^2(B_r)}^2
+\int_{B_r}\frac{\eps^2|\partial_xu_t^{\eps,r}|^2}{(|\partial_xu_t^{\eps,r}|^2+\eps^2)^{3/2}}\Big)\,\mathrm{d}x\\
&\leq\frac{1}{2}\|u_t^{\eps,r}(0)\|_{L^2(B_r)}^2\exp\big(\int_0^T
\|\partial_xu^{\eps,r}(t)\|_{L^\infty(B_r)}
\,\mathrm{d}t\big)\\
&\leq C(\eps,T)\,,
\end{split}
\end{equation}
thanks to estimate \eqref{uni-bd-sol-DCP-smooth6} and the uniform bound on the initial data \eqref{uni-bd-init-r2}. 
\\[1ex]
{\bf Proof of \eqref{uni-bd-sol-DCP-smooth8}: }We rewrite \eqref{DCP-smooth1} as
\begin{equation}
\label{uni-bd-sol-DCP-smooth81}
\Big(
\frac{\eps^2}{(|\partial_xu^{\eps,r}|^2+\eps^2)^{3/2}}+\eps
\Big)\partial_x^2u^{\eps,r}=\partial_tu^{\eps,r}+u^{\eps,r}\partial_xu^{\eps,r}\,.
\end{equation}
Then, thanks to estimates \eqref{uni-bd-sol-DCP-smooth7}, \eqref{uni-bd-sol-DCP-smooth1}, and \eqref{uni-bd-sol-DCP-smooth4} one can conclude from \eqref{uni-bd-sol-DCP-smooth81} for all $T'\leq T$ that also 
\begin{equation}
\label{uni-bd-sol-DCP-smooth82}
\|\partial_x^2u^{\eps,r}(T')\|_{L^2(B_r)}^2\leq\|\partial_tu^{\eps,r}(T')\|_{L^2(B_r)}^2+\|u^{\eps,r}(T')\|_{L^\infty(B_r)}^2
\|\partial_xu^{\eps,r}(T')\|_{L^2(B_r)}^2 \leq C(\eps,T)\,,
\end{equation}
which proves \eqref{uni-bd-sol-DCP-smooth8}. 
\\[1ex]
{\bf Proof of \eqref{uni-bd-sol-DCP-smooth9}: }Recall that
\begin{equation*}
\partial_xR^{\eps,r}=\partial_x\frac{\partial_xu^{\eps,r}}{\sqrt{|\partial_xu^{\eps,r}|^2+\eps^2}}
=\frac{\partial_x^2u^{\eps,r}}{\sqrt{|\partial_xu^{\eps,r}|^2+\eps^2}}-
\frac{\partial_xu^{\eps,r}}{\big(|\partial_xu^{\eps,r}|^2+\eps^2\big)^{3/2}}\,\partial_x^2u^{\eps,r}
\end{equation*}
In view of estimates \eqref{uni-bd-sol-DCP-smooth5} we therefore find
\begin{equation*}
\begin{split}
&\int_0^T\|\partial_xR^{\eps,r}\|_{L^2(B_r)}^2\,\mathrm{d}t
\leq\frac{2}{\eps^4}\int_0^T\|\partial_x^2u^{\eps,r}\|_{L^2(B_r)}^2\,\mathrm{d}t
+2\int\int_{B_r}\frac{|\partial_xu^{\eps,r}|^2}{(|\partial_xu^{\eps,r}|^2+\eps^2)^3}\,
|\partial_x^2u^{\eps,r}|^2\,\mathrm{d}x\,\mathrm{d}t\\
&\leq\frac{4}{\eps^4}\int_0^T\|\partial_x^2u^{\eps,r}\|_{L^2(B_r)}^2\,\mathrm{d}t
\leq C(\eps,T)\,,
\end{split}
\end{equation*}
which proves \eqref{uni-bd-sol-DCP-smooth9} and concludes the proof of the uniform a priori bounds \eqref{uni-bd-sol-DCP-smooth}.
\end{proof}
{\bf To Step 3: Limit passage $r\to\infty$: }
\begin{proposition}[Well-posedness of a global $H^2$ solution]
\label{Prop-limit-r}
Let the assumptions of Prop.\ \ref{Prop-uni-bd-r} be satisfied. Keep $\eps>0$ and $T>0$ fixed. Assume that the sequence of initial data 
$(u^{\eps,r}_\mathrm{in})_r$ is well-prepared, such that $u^{\eps,r}_\mathrm{in}\to u^{\eps}_\mathrm{in}$ in $H^2(\R)$ with 
\begin{equation}
u^{\eps}_\mathrm{in}\in H^2(\R)
\quad\text{ and }\quad
\lim_{|x|\to\infty}u^{\eps}_\mathrm{in}(x)=0\,.
\end{equation}
Then there exists a (not relabelled) subsequence $r\to\infty$ and a limit function $u^\eps:[0,T]\times\R\to\R$ such that, for every $r'>0$ 
fixed, the subsequence $(u^{\eps,r})_{r}$ converges to $u^\eps$ in the weak topology imposed by the uniform bounds 
\eqref{uni-bd-sol-DCP-smooth} 
locally on $B_{r'}$ as $r\to\infty$. The limit function $u^\eps$ has the following regularity
\begin{equation}
\label{reg-limit-r}
\begin{split}
u^\eps&\in L^\infty(0,T;H^2(\R))\cap L^2(0,T;H^2(\R))\quad\text{ with }\\
\partial_tu^\eps&\in L^\infty(0,T;L^2(\R))\cap L^2(0,T;H^1(\R))\,,
\end{split}
\end{equation}  
and it is a weak solution of  Cauchy problem \eqref{ibvp-approx}. In particular, $u^\eps$ satisfies the pointwise far-field boundary condition 
\eqref{bc:far-field}, i.e., 
\begin{equation}
\label{bc-far-field-pf}
\lim_{|x|\to\infty}u^\eps(x,t)=0\quad\text{ for all }t\in[0,T]\,.
\end{equation}
\end{proposition}   
\begin{proof}
The statements of the proposition directly follow from the a priori bounds \eqref{uni-bd-sol-DCP-smooth}, which are uniform in $r$. In particular, for $r_1>0$ fixed one finds a subsequence of solutions which converges in the topology imposed by \eqref{uni-bd-sol-DCP-smooth} on 
$B_{r_1}$ to a limit $u^\eps_{r_1}$. For any $r_2>r_1$ one finds a further subsequence  of solutions which converge in the topology imposed by \eqref{uni-bd-sol-DCP-smooth} on 
$B_{r_2}$ to a limit $u^\eps_{r_2}$ and by the uniqueness of the limit one concludes that $u^\eps_{r_2}\big|_{B_{r_1}}=u^\eps_{r_1}$. 
In this way, when selecting a sequence $r_n\to\infty$ as $n\to\infty,$ one finds a subsequence that converges in the topology imposed by \eqref{uni-bd-sol-DCP-smooth} on each 
$B_{r_n}$ to a limit $u^\eps_{r_n}$ such that $u^\eps_{r_n}\big|_{B_{r_m}}=u^\eps_{r_m}$ for every $r_m<r_n$. This results in the limit function $u^\eps:[0,T]\times\R\to\R,$ that inherits the regularity properties \eqref{reg-limit-r} from the uniform bounds \eqref{uni-bd-sol-DCP-smooth}. 
\par
It remains to verify that also the pointwise far-field boundary condition \eqref{bc-far-field-pf} holds true. 
For this we observe that the regularity \eqref{reg-limit-r} implies that $u^\eps\in \mathrm{C}(0,T;\mathrm{C}(\overline{\R})),$ i.e., in particular, 
for all $t>0$ the function $u^\eps(t,\cdot)$ is uniformly continuous in $\R$. Let now $t>0$ be fixed and assume that there is a subsequence $(x_j)_{j\in\N}$ with $x_j\to\infty$ and such that $u^\eps(t,x_j)\to c\not=0$ as $j\to\infty$. 
Thus, for any $\kappa>0$ sufficiently small such that $|c-\kappa/2|>\kappa/2,$ there is an index $j_\kappa\in\N$ such that for all $j\geq j_\kappa$ there holds $|u^\eps(t,x_j)- c|<\kappa/4$. 
By the uniform continuity of  $u^\eps(t,\cdot)$ there is $\delta_\kappa>0$ such that for all $j\geq j_\kappa$ and for all 
$x\in B_{\delta_\kappa}(x_j)$ there holds $|u^\eps(t,x)-u^\eps(t,x_j)|<\kappa/4$. Therefore we have 
\begin{equation*}
|u^\eps(t,x)|\geq |c-\kappa/2|>\kappa/2\quad\text{ for all $x\in B_{\delta_\kappa}(x_j)$ and for all $j\geq j_\kappa$}\,.
\end{equation*}
In view of regularity \eqref{reg-limit-r} this results the following contradiction:
\begin{equation*}
\infty>\|u^\eps(t)\|_{L^2(\R)}^2>\sum_{j=j_\kappa}^\infty\int_{B_{\delta_\kappa}(x_j)}|\kappa/2|^2\,\mathrm{d}x
=\sum_{j=j_\kappa}^\infty\frac{\kappa}{2}\mathcal{L}^1(B_{\delta_\kappa}(x_j))=\infty\,.
\end{equation*}
This shows that \eqref{bc-far-field-pf} must be satisfied. 
\end{proof}
{\bf To Step 4: Verification of the integral far-field relations for the limit solutions: }
\begin{proposition}[Integral far-field relations]
\label{Prop-Far-field}
Let the assumptions of Prop.\ \ref{Prop-limit-r} be satisfied. Keep $\eps>0$ and $T>0$ fixed. Then the weak solution $u^\eps$ of  Cauchy problem \eqref{ibvp-approx} also satisfies the integral far-field relations \eqref{int-far-field}, i.e., for all $t\in[0,T]$ there holds 
\begin{subequations}
\label{int-far-field-pf}
\begin{eqnarray}
\label{int-far-field1pf}
\int_0^T\tfrac{1}{3}u^\varepsilon(\pm r,t)^3\,\mathrm{d}t&\to&0\quad\text{as }r\to\infty\,,\\
\label{int-far-field2pf}
\int_0^Tu^\varepsilon(\pm r,t)\varepsilon \partial_xu^\varepsilon(\pm r,t)\,\mathrm{d}t&\to&0\quad\text{as }r\to\infty\,,\\
\label{int-far-field3pf}
\int_0^Tu^\varepsilon(\pm r,t)\frac{\partial_xu^\varepsilon(\pm r,t)}{\sqrt{|\partial_xu^\varepsilon(\pm r,t)|^2+\varepsilon^2}}\,\mathrm{d}t&\to&0\quad\text{as }r\to\infty\,.
\end{eqnarray}
\end{subequations}
\end{proposition}
\begin{proof}
We first show that 
\begin{equation}
\label{int-far-field-pf1}
\int_0^T|u^\varepsilon(\pm r,t)|^2\,\mathrm{d}t\to0\quad\text{as }r\to\infty\,.
\end{equation}
For this, we argue by contradiction and assume that there is a subsequence $r_j\to\infty$  such that 
\begin{equation}
\label{int-far-field-pf2}
\int_0^T|u^\varepsilon(r_j,t)|^2\,\mathrm{d}t\to c>0\quad\text{as }j\to\infty\,.
\end{equation}
Hence, there is an index $j_c>0$ such that for all $j\geq j_c$ we have $\int_0^T|u^\varepsilon(r_j,t)|\,\mathrm{d}t\geq c/4$. 
The regularity properties \eqref{reg-limit-r} imply that the map $x\mapsto\|u^\eps(\cdot,x)\|_{L^2(0,T)}^2$ is uniformly continuous. 
Thus, there is $\delta_c>0$ such that for all $j\geq j_c$ and for all $x\in B_{\delta_c}(r_j)$ we have $\|u^\eps(\cdot,x)\|_{L^2(0,T)}^2>c/2$. 
Again by the regularity properties \eqref{reg-limit-r} we therefore arrive at the following contradiction
\begin{equation}
\begin{split}
\infty>\|u^\eps\|_{L^2(0,T;H^2(\R))}^2&\geq\|u^\eps\|_{L^2(0,T;L^2(\R))}^2=\int_0^\infty\int_0^\infty|u^\eps|^2\,\mathrm{d}t\,\mathrm{d}x\\
&\geq\sum_{j=j_c}^\infty\int_{B_{\delta_c}(r_j)}\int_0^\infty|u^\eps|^2\,\mathrm{d}t\,\mathrm{d}x
\geq\sum_{j=j_c}^\infty\frac{c}{2}\mathcal{L}^1(B_{\delta_c}(r_j))=\infty\,.
\end{split}
\end{equation}
Hence, \eqref{int-far-field-pf1} holds true. 
\par 
In order to deduce \eqref{int-far-field1pf} we use that \eqref{reg-limit-r} implies that $\|u^\eps\|_{L^\infty(0,T;L^\infty(\R)}\leq C$ and therefore we find with the aid of H\"older's inequality
\begin{equation}
\Big|\int_0^T\tfrac{1}{3}u^\varepsilon(\pm r,t)^3\,\mathrm{d}t\Big|
\leq\int_0^T\tfrac{1}{3}|u^\varepsilon(\pm r,t)|^3\,\mathrm{d}t
\leq \tfrac{1}{3}\|u^\eps\|_{L^\infty(0,T;L^\infty(\R)}\int_0^T|u^\varepsilon(\pm r,t)|^2\,\mathrm{d}t\to0
\end{equation}
as $r\to\infty$ thanks to \eqref{int-far-field-pf1}. 
Similarly we also conclude \eqref{int-far-field2pf}, i.e., 
\begin{align*}
\Big|\int_0^Tu^\varepsilon(\pm r,t)\varepsilon \partial_xu^\varepsilon(\pm r,t)\,\mathrm{d}t\Big|
&\leq\int_0^T|u^\varepsilon(\pm r,t)\varepsilon \partial_xu^\varepsilon(\pm r,t)|\,\mathrm{d}t\\
&\leq\Big(\int_0^T|\eps\partial_xu^\varepsilon(\pm r,t)|^2\,\mathrm{d}t\Big)^{1/2}\Big(\int_0^T|u^\varepsilon(\pm r,t)|^2\,\mathrm{d}t\Big)^{1/2}\to0\,,
\end{align*}
as well as \eqref{int-far-field3pf}, i.e., 
\begin{align*}
&\Big|\int_0^Tu^\varepsilon(\pm r,t)\frac{\partial_xu^\varepsilon(\pm r,t)}{\sqrt{|\partial_xu^\varepsilon(\pm r,t)|^2+\varepsilon^2}}\,\mathrm{d}t\Big|
\leq\int_0^T|u^\varepsilon(\pm r,t)|\frac{|\partial_xu^\varepsilon(\pm r,t)|}{\sqrt{|\partial_xu^\varepsilon(\pm r,t)|^2+\varepsilon^2}}\,\mathrm{d}t\\
&\leq\int_0^T|u^\varepsilon(\pm r,t)|\frac{|\partial_xu^\varepsilon(\pm r,t)|}{\varepsilon}\,\mathrm{d}t
\leq\frac{1}{\eps}\Big(\int_0^T|\eps\partial_xu^\varepsilon(\pm r,t)|^2\,\mathrm{d}t\Big)^{1/2}\Big(\int_0^T|u^\varepsilon(\pm r,t)|^2\,\mathrm{d}t\Big)^{1/2}\to0\,,
\end{align*}
again by regularity properties \eqref{reg-limit-r} and \eqref{int-far-field-pf1}. 
\end{proof}

\paragraph{Acknowledgements: }M.T.\ and E.S.T.\ are grateful for the support by Deutsche Forschungsgemeinschaft (DFG) within CRC 1114 \emph{Scaling Cascades in Complex Systems}, Project-Number 235221301, in particular within the Projects B09 \emph{Materials with Discontinuities on Many Scales} and C09 \emph{Dynamics of Rock Dehydration on Multiple Scales}. M.T.\ also acknowledges the support by DFG within CRC TRR 388 \emph{Rough Analysis, Stochastic Dynamics \& Related Fields}, Project-Number 516748464, Project B09 
\emph{Mean field theories and scaling limits of nonlinear stochastic evolution systems}.  M.T.'s work also benefitted from the kind hospitality and the inspiring research environment at the Department of Mathematics at Texas A\&M University and at the Department of Applied Mathematics and Theoretical Physics, University of Cambridge, where part of this work was completed. Moreover, the work of E.S.T.\ was also supported in part by the DFG Research Unit FOR 5528 on Geophysical Flows.  

\bibliographystyle{alpha}
\bibliography{Library.bib}
\end{document}